\documentclass[11pt]{article}
\usepackage{amsfonts}
\usepackage{mathrsfs}
\usepackage{bbm}
\usepackage{indentfirst}
\usepackage{amssymb,amsmath,amsthm,graphicx}
\usepackage{float}
\usepackage[colorlinks, linkcolor=red]{hyperref}
\usepackage{color, xcolor}
\usepackage{cite}
\usepackage{enumerate}
\usepackage{esint}
\usepackage[paper=a4paper, left=1.6cm, right=1.6cm, top=1.8cm, bottom=1.6cm, headheight=5.5pt, footskip=0.8cm, footnotesep=0.8cm, centering, includefoot]{geometry}

\hypersetup{urlcolor=red, citecolor=red}
\usepackage[dvips]{epsfig}
\usepackage{amscd}

\usepackage{amsthm}
\usepackage{mathrsfs}
\usepackage{verbatim}

\newtheorem{theorem}{Theorem}[section]

\newtheorem{definition}{Definition}[section]

\newtheorem{lemma}{Lemma}[section]

\newtheorem{remark}{Remark}[section]

\newcommand{\eqnum}{\refstepcounter{equation}\textup{\labelform@{\theequation}}}
\newcommand{\dis}{\displaystyle}

\newcommand{\No}{{\bf n}}
\newcommand{\Ta}{{\bf{t}}}

\newcommand{\cK}{\mathcal K}
\newcommand{\cL}{\mathcal L}

\newcommand{\R}{\mathbb R}

\newcommand{\loc}{\text{\rm loc}}

\newcommand{\intoxb}{\int_{\partial\Omega}}
\newcommand{\intox}{\int_{\Omega}}
\newcommand{\intoxt}{\int_0^t\int_{\Omega}}
\newcommand{\intoxT}{\int_0^T\int_{\Omega}}
\newcommand{\intoxtb}{\int_0^t\int_{\partial\Omega}}
\newcommand{\trho}{\tilde\rho}
\newcommand{\brho}{\bar\rho}
\newcommand{\tu}{\tilde{{\bf u}}}
\newcommand{\tX}{\tilde{X}}
\newcommand{\vu}{{\bf u}}
\newcommand{\vw}{{\bf w}}
\newcommand{\vvv}{{\bf v}}

\newcommand{\supp}{\text{\rm supp }}
\newcommand{\seta}{\tilde\sum_1^{M,T}}
\newcommand{\setb}{\tilde\sum_2^{N,T}}
\DeclareMathOperator{\divv}{div}

\begin{document}
\title{Large friction limit of compressible Navier--Stokes equations with Navier boundary conditions in a half-space
\thanks{
Suen's research was partially supported by Hong Kong General Research Fund (No. 18300424). Zhong's research was partially supported by National Natural Science Foundation of China (No. 12371227) and Fundamental Research Funds for the Central Universities (No. SWU--KU24001).
}
}

\author{Anthony Suen$\,^{\rm 1}\,$,\ Xin Zhong$\,^{\rm 2}\,$ {\thanks{E-mail addresses: acksuen@eduhk.hk (A. Suen), xzhong1014@amss.ac.cn (X. Zhong).}}\date{}\\
\footnotesize $^{\rm 1}\,$
Department of Mathematics and Information Technology, The Education University of Hong Kong, Hong Kong\\
\footnotesize $^{\rm 2}\,$ School of Mathematics and Statistics, Southwest University, Chongqing 400715, P. R. China}

\renewcommand{\theequation}{\thesection.\arabic{equation}}
\numberwithin{equation}{section}
\numberwithin{theorem}{section}
\numberwithin{figure}{section}
\maketitle
\begin{abstract}
We study the large-friction limit for the three-dimensional barotropic
compressible Navier--Stokes equations in a half-space. The velocity
satisfies a Navier boundary condition with friction coefficient
$\alpha>0$, while the limiting problem satisfies the no-slip boundary
condition. We establish estimates for local-in-time smooth solutions
that are uniform in $\alpha$ and prove strong convergence of the density
and velocity as $\alpha\to\infty$. For weak solutions, we use the
Lagrangian flow maps associated with the two velocities to compare the
densities and construct suitable transported test functions. This yields
weak convergence to the no-slip solution. Our results provide a
compressible counterpart of the large-friction limit for incompressible
flows.
\end{abstract}



\tableofcontents

\section{Introduction}\label{Introduction section}

\subsection{Background and motivation}


The Navier--Stokes equations of a compressible barotropic fluid in the half-space $\Omega=\{x\in\R^3:x_3>0\}$ can be given by the conservation of mass and the balance of momentum:
\begin{align}\label{Navier Stokes}
\begin{cases}
\rho_t + \divv (\rho \vu) =0, \\
(\rho u^j)_t + \divv (\rho u^j \vu) + (P(\rho))_{x_j} = \mu\,\Delta u^j +
\lambda \, (\divv \,\vu)_{x_j}.
\end{cases}
\end{align}
The system \eqref{Navier Stokes} is supplemented with the initial data
\begin{align}\label{initial condition}
(\rho(\cdot,0),\vu(\cdot,0))=(\rho_0,\vu_0)
\end{align}
and {\it either one} of the following boundary conditions, namely
\begin{align}\label{Navier condition}
\mbox{$\vu\cdot \No(x)=0$ and $(u^1(x),u^2(x),u^3(x))=K\mu(u^1_{x_3}(x),u^2_{x_3}(x),0)$ for $x\in\partial\Omega$}
\end{align}
or
\begin{align}\label{no slip condition}
\vu(x)=0,\qquad\mbox{for $x\in\partial\Omega$},
\end{align}
where $\partial\Omega=\R^3_{+}$. The meanings for the functions and symbols are given as follows:
\begin{itemize}
\item $\rho$ and $\vu=(u^1,u^2,u^3)$ are functions of $x\in\Omega$ and $t\ge0$ which represent density and velocity respectively;
\item $P = P(\rho)$ is a given function in $\rho$ which stands for the pressure;
\item $\mu$, $\lambda > 0$ are viscosity constants;
\item $\vu_{\tau}=(u^{1},u^{2})$ is the tangential component of the velocity $\vu$ on the boundary $\partial\Omega$;
\item $(\cdot)_{x_j}$ and $(\cdot)_t$ stand for the spatial derivative $\partial_{x_j}$ and time derivative $\partial_{t}$ respectively;
\item $\divv(\cdot)$ and $\Delta(\cdot)$ are the usual spatial divergence and Laplace operators;
\item $\nabla_{\tau}(\cdot)=(\partial_{x_1},\partial_{x_2})$ is the gradient operator with respect to $x_1$ and $x_2$;
\item $\No(x)=(0,0,-1)$ and $\Ta(x)$ are the unit outward normal and tangent vectors respectively on $\partial \Omega$;
\item $K>0$ is a constant and $\alpha:=K^{-1}$ is known as the {\it Navier friction coefficient}.
\end{itemize}
To facilitate later discussions, we name \eqref{Navier Stokes} and \eqref{initial condition} with boundary condition \eqref{Navier condition} the {\it Navier--Stokes equations with Navier boundary condition} (NSENBC), while we name \eqref{Navier Stokes} and \eqref{initial condition} with boundary condition \eqref{no slip condition} the {\it Navier--Stokes equations with Dirichlet boundary condition} (NSEDBC). The sign of the Navier boundary condition is important in the energy method. With the convention $\No=(0,0,-1)$, the boundary contribution generated by integration by parts has the dissipative form
\begin{align*}
\alpha\|\vu_{\tau}\|_{L^2(\partial\Omega)}^2,
\qquad \alpha=K^{-1}.
\end{align*}
This nonnegative term is essential for obtaining estimates whose
constants are independent of $K$.

The so-called Navier boundary condition \eqref{Navier condition} was first proposed by Navier in \cite{Navier23} which states that the velocity on $\partial\Omega$ is proportional to the tangential component of the stress, while the Dirichlet boundary condition \eqref{no slip condition} (or more precisely the {\it no-slip} boundary condition) assumes the fluid will have zero velocity relative to the boundary. The Navier boundary condition interpolates between tangential slip and
no-slip behaviour through the parameter $K$. When $K$ is positive, the
velocity is allowed to have a nonzero tangential trace on the boundary,
while the normal component vanishes. As $K$ decreases, the tangential
friction becomes stronger and the boundary condition formally approaches
the no-slip condition. For compressible flow, however, this formal
observation does not by itself imply convergence of solutions since the
density is transported by the velocity and the two problems generate
different boundary terms in their weak formulations. The limit $K\to0$ is equivalent to the large-friction limit
$\alpha=K^{-1}\to\infty$. In this limit, the tangential component of
the velocity is increasingly penalised at the boundary and the Navier
boundary condition formally approaches the no-slip condition. Therefore, regarding the two different boundary conditions given by \eqref{Navier condition} and \eqref{no slip condition}, a very interesting but natural question arises:
\begin{itemize}
\item[$(\mathcal{Q})$:] {\it As the constant $K$ vanishes (or equivalently the Navier friction coefficient $\alpha$ tends to infinity), will the solutions to the Navier--Stokes equations with Navier boundary conditions converge to a solution to the Navier--Stokes equations with the usual no-slip boundary conditions?}
\end{itemize}
The present work is devoted to answering the question $(\mathcal{Q})$ by justifying the limit for local-in-time solutions of \eqref{Navier Stokes}--\eqref{initial condition} with boundary condition \eqref{Navier condition} as $K$ tends to zero.

We first briefly recall some known existence results regarding the problem \eqref{Navier Stokes}--\eqref{initial condition}. Generally speaking, by imposing different conditions on the initial data, three types of solutions can be shown to exist and they enjoy different properties:
\begin{itemize}
\item[(i)] When the initial data are taken to be close to constants in $H^3(\Omega)$, Matsumura and Nishida \cite{MN80,MN83} obtained global-in-time ``small-smooth type'' solutions, which were later extended by Danchin \cite{danchin00, danchin01} to solutions in certain scale-invariant homogeneous Besov spaces. Small-smooth type solutions are constructed via iterative procedure based on asymptotic decay rates for the corresponding linearised system, and they do not exhibit the generic singularities of the system.
\item[(ii)] When the initial data have nonnegative density with arbitrarily large energy, Lions \cite{lions98} and Feireisl \cite{feireisl04} proved the global existence of ``large-weak type'' solutions. Large-energy weak solutions by their very nature possess very little regularity, which may even include some non-physical solutions (see \cite{hoffserre91} for related discussions).
\item[(iii)] Different from two types of solutions mentioned in (i) and (ii), a third type of ``intermediate-weak type'' solutions were studied by Hoff \cite{hoff95, hoff05}, Perepelitsa \cite{perepelitsa14}, and Suen \cite{suen13, suen14, suen21} for which initial data are small in fairly weak norms and initial densities are nonnegative and essentially bounded. Such intermediate-weak type solutions have rich physical and mathematical meanings compared to other solution classes: on the one hand, these solutions may exhibit discontinuities in density and velocity gradient across hypersurfaces in $\R^3$, which is not observable from small-smooth solutions; on the other hand, the solutions would still have enough regularity for the development of a uniqueness and continuous dependence theory \cite{cheungsuen16, hoff06, suen20a} which seems difficult from the very weak framework for large-weak type solutions. Moreover, such solutions demonstrate fine-structure property near the boundary under no-slip boundary condition \eqref{no slip condition} (refer to \cite{hoffperepelitsa12} for the instantaneous tangency for density interfaces), which behave drastically different from the case when the Navier boundary condition \eqref{Navier condition} is imposed (see \cite{hoff05}).
\end{itemize}

The convergence issue raised by question $(\mathcal{Q})$ was studied by Kelliher \cite{kelliher06} and Kim \cite{kim09}, where the authors proved the convergence of smooth solutions as $K$ vanishes for incompressible Navier--Stokes equations. In contrast to the incompressible flows, the compressible system \eqref{Navier Stokes} contains the unknown density function $\rho$, which contributes to the following difficulties:
\begin{itemize}
\item one has to obtain {\it appropriate} bounds on $\rho$ in order to gain control over the velocity $\vu$, which is different from the incompressible case for which one can automatically have some control over $\vu$ on the boundary (see \cite[Lemma~9.1]{kelliher06});
\item one has to ensure some {\it $K$-independent} bounds on both $\rho$ and $\vu$ before taking the limit $K\to0$, which is highly non-trivial compared with the incompressible cases;
\item one has to estimate the {\it difference} of weak solutions with low regularity, which can only be accomplished by considering the weak formulations of \eqref{Navier Stokes}--\eqref{initial condition} with different boundary conditions.
\end{itemize}

The large-friction limit has been studied for incompressible
Navier--Stokes equations, while Hoff established weak solutions for
compressible Navier--Stokes equations with a fixed Navier boundary
condition. The problem considered here is different: we study the
behavior of compressible solutions as the friction coefficient
$\alpha=K^{-1}$ tends to infinity and compare the limiting problem with
the no-slip problem. In the present work, we study the vanishing-slip-length limit for the
three-dimensional compressible Navier--Stokes equations in the half-space.
More precisely, we compare solutions satisfying the Navier boundary
condition \eqref{Navier condition} with solutions satisfying the no-slip condition \eqref{no slip condition}. The main contributions of this paper are as follows:

\begin{itemize}
    \item We establish local-in-time estimates for solutions of the
    compressible Navier--Stokes equations with Navier boundary conditions
    whose constants are independent of the slip length $K$. This
    uniformity is essential because the boundary condition becomes
    singular in the limit $K\to0$.

    \item We prove elliptic estimates for the Lam\'e operator with the
    mixed normal and tangential boundary conditions induced by the Navier
    condition. In particular, the estimates provide the regularity and
    boundary control required in the subsequent compactness and
    comparison arguments.

    \item For smooth solutions, we prove strong convergence of the
    density and velocity of the Navier problem to those of the
    corresponding no-slip problem as $K\to0$. This extends the
    vanishing-slip-length analysis from incompressible flow to the
    compressible setting, where the density must be controlled together
    with the velocity.

    \item For weak solutions, we develop a Lagrangian comparison argument.
    The flow maps generated by the Navier and no-slip velocities are used
    to compare the transported densities and to construct suitable
    transported test functions in the momentum equations.

    \item We explicitly treat the boundary mismatch between the two
    problems. After subtracting the weak formulations, the boundary
    contributions are separated into the term involving the Navier
    solution, which cancels by the adjoint Navier boundary condition, and
    the terms involving the no-slip solution, which are estimated by the
    uniform boundary and trace bounds.
\end{itemize}

The use of a Lagrangian framework for comparison and continuous
dependence of weak compressible flows goes back to Hoff \cite{hoff06}; here it is
adapted to the comparison of two solutions satisfying different
boundary conditions. The half-space geometry is important because it
allows us to control the flow maps, tangential derivatives, traces and
boundary changes of variables in a unified manner \cite{hoff12}.

There are three main analytical difficulties in this problem:
\begin{itemize}
\item First, the limit $K\to0$ is singular at the boundary. In terms of the
Navier friction coefficient $\alpha=K^{-1}$, the boundary dissipation
takes the form $\alpha\|\vu_{\tau}\|_{L^2(\partial\Omega)}^2$. Although this term is nonnegative in the energy estimate, higher-order
estimates must be arranged so that their constants remain independent of
$\alpha$. In particular, the elliptic estimates for the Lam\'e system
must be compatible with the simultaneous normal and tangential boundary
conditions.
\item Second, compressibility couples the density and velocity estimates. The
density satisfies the transport equation \eqref{Navier Stokes}$_1$, so the density difference depends on the difference of the velocity
fields and on their divergences. Consequently, convergence of the
velocity alone is not sufficient; the transport structure and the
associated flow maps must also be controlled.
\item Third, the Navier and no-slip solutions cannot be compared by a direct
subtraction of their weak formulations without generating nontrivial
boundary terms. The admissible test functions must satisfy the normal
boundary constraint, while the tangential components are governed by
different boundary conditions for the two solutions. The comparison
therefore requires transported test functions and a careful decomposition
of the boundary contributions. The part involving the Navier solution
cancels through the adjoint Navier condition, whereas the remaining
part involving the no-slip solution is controlled by trace estimates.
\end{itemize}

\subsection{Main results}

We now give a precise formulation of our results. The parameters $\mu$, $\lambda$, and $K$ are assumed to satisfy
\begin{gather}\label{conditions on viscosity}
\mbox{$\mu>0$ and $0<4\lambda<5\mu$;}
\\ \label{condition on K}
\mbox{$K>0$ is independent of $x$ and $t$.}
\end{gather}
and the pressure $P$ is assumed to verify
\begin{gather}\label{condition on pressure}
P(\cdot)\in C[0, \infty)\cap C^2(0,\infty).
\end{gather}
We let $\brho$ be a fixed, positive, constant reference density and define $\bar P=P(\brho)$.

This paper contains two types of convergence results. For smooth
solutions, the uniform estimates allow a direct comparison of the
Navier and no-slip problems and yield strong convergence. For weak
solutions, the available regularity is insufficient for the same
argument. We therefore use the Lagrangian flow maps generated by the
two velocity fields and obtain convergence in the weaker topology
stated in Theorem~\ref{convergence weak thm}.

To begin with, we have the following local-in-time existence theorem for {\it smooth solutions} to \eqref{Navier Stokes}. It gives the necessary bounds on the smooth solutions and will be useful for proving the convergence results later.

\begin{theorem}\label{local-in-time existence theorem}
Assume that the hypotheses \eqref{conditions on viscosity}--\eqref{condition on pressure} are hold and let positive numbers $M_0$ and $M_0'>M_0$ be given. Then there are a positive time $T$ and a positive constant $N$ depending on $\Omega$, $\mu$, $\lambda$, $P$, $M_0$, and $M_0'$ but independent of $K$ such that if the initial data $(\rho_0,\vu_0)$ satisfies
\begin{equation}\label{boundedness condition on initial data}
\|\rho_0\|_{L^\infty},\|\rho_0^{-1}\|_{L^\infty},\|\rho_0-\brho\|_{H^2},\|\vu_0\|_{H^3}\le M_0,
\end{equation}
\begin{equation}\label{compatibility condition on u0 no-slip}
\vu_0(x)=0,\quad x\in\partial\Omega,
\end{equation}
there are local-in-time smooth solutions $(\rho, \vu)$ and $(\trho,\tu)$ to the initial-boundary value problems (NSENBC) and (NSEDBC) respectively which are defined in $\Omega \times [0, T]$ and they satisfy
\begin{align}\label{boundedness on rho local-in-time solution}
&\sup_{0\le t\le T} (\|\rho(\cdot,t)\|_{L^\infty},\|\rho^{-1}(\cdot,t)\|_{L^\infty},\|\rho(\cdot,t)-\brho\|_{H^2})\notag\\
&\quad+\sup_{0\le t\le T} (\|\trho(\cdot,t)\|_{L^\infty},\|\trho^{-1}(\cdot,t)\|_{L^\infty},\|\trho(\cdot,t)-\brho\|_{H^2})\le M_0'
\end{align}
and
\begin{align}\label{boundedness on u local-in-time solution}
&\sup_{0\le t\le T} (\|\vu(\cdot,t)\|_{H^3}^2,\|\vu_t(\cdot,t)\|_{H^1}^2)+\int_0^T\|\vu_t(\cdot,t)\|_{H^2}^2dt\notag\\
&\quad+\sup_{0\le t\le T} (\|\tu(\cdot,t)\|_{H^3}^2,\|\tu_t(\cdot,t)\|_{H^1}^2)+\int_0^T\|\tu_t(\cdot,t)\|_{H^2}^2dt\le N.
\end{align}
\end{theorem}

\begin{remark}
Here are some remarks regarding Theorem~\ref{local-in-time existence theorem}.
\begin{itemize}
\item In view of Theorem~\ref{local-in-time existence theorem}, for the case of (NSENBC), one can replace the compatibility condition \eqref{compatibility condition on u0 no-slip} on $\vu_0$ with a weaker assumption
\begin{align*}
\mbox{$\vu\cdot \No_0(x)=0$ and $K\mu(\nabla \vu_0\No)\cdot \Ta(x)=- \vu_0\cdot \Ta(x)$, \qquad for $x\in\partial\Omega$.}
\end{align*}
\item The key idea of proving Theorem~\ref{local-in-time existence theorem} follows closely to the one given in \cite{hoff12}, the main difference here is to show the $K$-independence of the constants $M'_0$ and $N$.
\item It is important to have $M'_0$ and $N$ both being independent of $K$, which allow us to apply the bounds \eqref{boundedness on rho local-in-time solution} and \eqref{boundedness on u local-in-time solution} for proving the convergence of smooth solutions as $K$ vanishes.
\end{itemize}
\end{remark}

The next result will treat the large friction limit of smooth solutions established in Theorem \ref{local-in-time existence theorem} as $K$ vanishes.
\begin{theorem}\label{convergence smooth thm}
Assume that the hypotheses \eqref{conditions on viscosity}--\eqref{condition on pressure} hold and let $(\rho_0,\vu_0)$ be given functions satisfying \eqref{boundedness condition on initial data} and \eqref{compatibility condition on u0 no-slip}. Suppose that $(\rho,\vu)$ and $(\trho,\tu)$ are smooth local-in-time solutions to the systems (NSENBC) and (NSEDBC) respectively defined in $\Omega\times[0,T]$ with the same initial data $(\rho_0,\vu_0)$ as described by Theorem~\ref{local-in-time existence theorem}. Then there exists $T^*\in(0,T]$ such that, for any $s_1\in[0,2]$ and $s_2\in[0,1]$,
\begin{gather}\label{convergence of u in Hs}
\vu\to\tu\mbox{ in $L^\infty([0,T^*];H^{s_1}(\Omega))$ as $K\to0$,}
\\ \label{convergence of u on the boundary}
\vu\to\tu\mbox{ in $L^2([0,T^*];L^2(\partial\Omega))$ as $K\to0$,}
\\ \label{convergence of rho in Hs}
\rho-\brho\to\trho-\brho\mbox{ in $L^\infty([0,T^*];H^{s_2}(\Omega))$ as $K\to0$.}
\end{gather}
\end{theorem}

\begin{remark}
As a consequence of Theorem~\ref{convergence smooth thm}, if we define the energy functional $E(\rho,\vu,t)$ by
\begin{align*}
E(\rho,\vu,t)=\sup_{0\le s\le t}\left[\int_{\Omega}(|\vu|^2+|\rho-\brho|^2)(x,s)dx\right]+\int_0^t\int_{\Omega}|\nabla \vu|^2(x,s)dxds,
\end{align*}
then, for all $t\in[0,T^*]$, we have $E(\rho,\vu,t)\to E(\trho,\tu,t)$ as $K\to0$.
\end{remark}

After we have addressed the smooth solutions to the systems (NSENBC) and (NSEDBC), we switch our attention to {\it weak solutions} that will be defined in the following sense. We say that $(\rho,\vu)$ is a weak solution of the system (NSENBC) on $\Omega\times[0,T]$ if:
\begin{itemize}
\item The following regularity conditions hold for $\rho$ and $\vu$
\begin{equation}\label{condition on weak sol1}
\mbox{$\rho-\brho\in L^{\infty}([0,T];L^1_{\loc}\cap H^{-1})$ and $\rho\ge0$ a.e.;}
\end{equation}
\begin{equation}\label{condition on weak sol2}
\rho \vu,P(\rho)-\bar P,\nabla \vu\in L^2(\Omega\times(0,T));\ \ \rho|\vu|^2\in L^1(\Omega\times(0,T));
\end{equation}
\item The weak form of the mass equation
\begin{align}\label{WF1}
\left.\int_{\Omega}\rho(x,\cdot)\varphi(x,\cdot)dx\right|_{t_1}^{t_2}=\int_{t_1}^{t_2}\int_{\Omega}(\rho\varphi_t + \rho \vu\cdot\nabla\varphi)dxd\tau
\end{align}
holds for all $T\ge t_2\ge t_1 \ge 0$ and $C^1$ test functions $\varphi\in\mathcal{D}(\Omega\times(-\infty,\infty))$ which are Lipschitz on $\Omega\times[t_1,t_2]$ with $\supp\varphi(\cdot,\tau)\subset \cK$, $\tau\in[t_1,t_2]$, where $\cK$ is compact;
\item The weak form of the momentum equation
\begin{align}\label{WF2}
\left.\int_{\Omega}(\rho u^{j})(x,\cdot)\psi(x,\cdot)dx\right|_{t_1}^{t_2}=&\int_{t_1}^{t_2}\int_{\Omega}[\rho u^{j}\psi_t + \rho u^{j}\vu\cdot\nabla\psi + (P(\rho)-\bar P)\psi_{x_j}]dxd\tau\notag\\
& - \int_{t_1}^{t_2}\int_{\Omega}[\mu\nabla u^{j}\cdot\nabla\psi + \lambda\,\divv(\vu)\psi_{x_j}]dxd\tau\notag\\
& - \mu\int_{t_1}^{t_2}\intoxb K^{-1}\vu\cdot\psi dS_{x}d\tau
\end{align}
holds for all $T\ge t_2\ge t_1 \ge 0$ and for test functions $\psi$ which are locally Lipschitz on $\Omega \times [t_1, t_2]$ with $\psi,\psi_t,\nabla\psi \in L^2(\Omega \times (t_1,t_2))$, $\nabla\psi \in L^\infty(\Omega \times (t_1,t_2))$, $\psi(x,t_2) = 0$ on $\bar{\Omega}$ and $\psi\cdot\No(x)=0$ on $\partial\Omega$.
\item The following boundary condition holds for all $x\in\partial\Omega$ and $t>0$
\begin{align}\label{NBC weak form}
(u^1(x),u^2(x),u^3(x,t))=K\mu(u^1_{x_3}(x,t),u^2_{x_3}(x,t),0).
\end{align}
\end{itemize}

Similarly, a weak solution of the system (NSEDBC) on $\Omega\times[0,T]$ is a pair $(\tilde\rho,\tilde\vu)$ satisfying
\begin{itemize}
\item the regularity conditions \eqref{condition on weak sol1}--\eqref{condition on weak sol2};
\item the weak form \eqref{WF1} of the mass equation;
\item  the boundary condition \eqref{no slip condition} (holds pointwisely) and
\item the following weak form of the momentum equation for all $T\ge t_2\ge t_1 \ge 0$
\begin{align}\label{WF2 DBC}
\left.\int_{\Omega}(\tilde\rho \tilde u^{j})(x,\cdot)\psi(x,\cdot)dx\right|_{t_1}^{t_2}=&\int_{t_1}^{t_2}\int_{\Omega}[\tilde\rho \tilde u^{j}\psi_t + \tilde\rho \tilde u^{j}\tilde\vu\cdot\nabla\psi + (P(\tilde\rho)-\bar P)\psi_{x_j}]dxd\tau\notag\\
& - \int_{t_1}^{t_2}\int_{\Omega}[\mu\nabla \tilde u^{j}\cdot\nabla\psi + \lambda\,\divv(\tilde\vu)\psi_{x_j}]dxd\tau\\
& + \int_{t_1}^{t_2}\int_{\partial\Omega}\lambda\nabla \tilde{u}^3\cdot\psi dS_{x}d\tau + \int_{t_1}^{t_2}\int_{\partial\Omega}\mu\partial_{x_3}\tu\cdot\psi dS_{x}d\tau\notag.
\end{align}
\end{itemize}

\begin{remark}
Here are some remarks on the existence of weak solutions:
\begin{itemize}
\item For each $K>0$, the existence of weak solution to the system (NSENBC) was proved by Hoff \cite{hoff05}, namely there exists a positive time $T>0$ and a weak solution to (NSENBC) defined on $\Omega\times[0,T]$ which satisfies conditions \eqref{condition on weak sol1}--\eqref{WF2}.
\item The existence of weak solution to the system (NSENBC) was proved by Perepelitsa \cite{perepelitsa14}, namely there exists a positive time $T>0$ and a weak solution to (NSEDBC) defined on $\Omega\times[0,T]$ which satisfies conditions \eqref{condition on weak sol1}--\eqref{WF1} and \eqref{WF2 DBC}.
\end{itemize}
\end{remark}

The constants appearing in the following assumptions are required to
be independent of $K$. This uniformity is the key hypothesis for the
large-friction limit. For each $K>0$, let $(\rho,\vu)$ and $(\trho,\tu)$ be weak solutions respectively to the systems (NSENBC) and (NSEDBC) both defined on $\Omega \times [0, T]$, where $T$ is taken to be independent of $K$. The two solutions $(\trho,\tu)$ and $(\rho,\vu)$ we consider will also be assumed to satisfy
\begin{align}\label{new spaces for weak sol 1}
\begin{cases}
\vu\in C(\Omega\times(0,T])\cap L^1((0,T_K);W^{1,\infty}(\bar{\Omega}))\cap L^\infty_{loc}((0,T];L^\infty(\Omega))\\
\tu\in C(\Omega\times(0,T])\cap L^1((0,T);W^{1,\infty}(\bar{\Omega}))\cap L^\infty_{loc}((0,T];L^\infty(\Omega)),
\end{cases}
\end{align}
\begin{align}\label{new spaces for weak sol 1'}
\int_0^{T}\|\nabla \vu(\cdot,\tau)\|_{L^\infty(\bar\Omega)}d\tau\le \bar{C}_1,\quad \int_0^{T}\|\nabla \tu(\cdot,\tau)\|_{L^\infty(\bar\Omega)}d\tau\le \bar{C}_2,
\end{align}
\begin{align}\label{spaces for weak sol 2}
\begin{cases}
\rho-\brho,\vu\in L^2(\Omega\times(0,T))\\
\trho-\brho,\tu\in L^2(\Omega\times(0,T)),
\end{cases}
\end{align}
\begin{gather}\label{spaces for weak sol 3}
\mbox{$\rho,\rho^{-1},\trho,\trho^{-1}\in [\bar{\rho}_1,\bar{\rho}_2]$ in $\Omega\times(0,T)$,}
\\ \label{spaces for weak sol bdry}
\int_0^T\int_{\partial\Omega}|\nabla\tu|^2 dS_{x}d\tau<\bar{C}_3,
\end{gather}
and
\begin{align}\label{spaces for weak sol 5}
&\int_0^T \tau^\frac{4}{5}\|\nabla \tilde F(\cdot,\tau)\|^{\frac{8}{5}}_{L^4}d\tau+\int_0^T\tau^\frac{4}{5}
\|\nabla\tilde\omega(\cdot,\tau)\|^{\frac{8}{5}}_{L^4}d\tau \notag \\
&\quad+\int_0^T[\|\tu(\cdot,\tau)\|^2_{L^\infty}+\tau\|\nabla\tu(\cdot,\tau)
\|^2_{L^\infty}]d\tau<\bar{C}_4,
\end{align}
where $\bar{C}_1,\bar{C}_2,\bar{C}_3,\bar{C}_4,\bar{\rho}_1$, and $\bar{\rho}_2$ are positive constants independent of $K$. The terms $\tilde{F}$ and $\tilde\omega$ are known as effective viscous flux and vorticity respectively, which are defined by
\begin{align}\label{def of F and omega}
\tilde{F}=(\mu+\lambda)\divv(\tilde{\vu})+P(\bar{\rho})-P(\tilde{\rho}),\quad \tilde{\omega}^{j,k}=\tilde{u}^j_{x_k}-\tilde{u}^k_{x_j}.
\end{align}

With respect to the initial data, suppose that the pairs of functions $(\rho_0,\vu_0)$ and $(\trho_0,\tu_0)$ satisfy \eqref{Navier condition} and \eqref{no slip condition} respectively with
\begin{align}\label{initial condition on both systems}
(\trho(\cdot,0),\tu(\cdot,0))=(\trho_0,\tu_0),\quad (\rho(\cdot,0),\vu(\cdot,0))=(\rho_0,\vu_0),
\end{align}
\begin{equation}\label{initial data being L2}
(\trho_0\tu_0,\rho_0\vu_0,\trho_0-\rho_0)\in L^2.
\end{equation}

We prove convergence for weak solutions belonging to a class satisfying the additional regularity and uniformity assumptions \eqref{new spaces for weak sol 1}--\eqref{spaces for weak sol 5}.

\begin{theorem}\label{convergence weak thm}
Let $\lambda$ and $\mu$ be the system parameters in \eqref{Navier Stokes} satisfying \eqref{conditions on viscosity}, and let $P$ satisfy
\begin{align}\label{extra condition on P}
P(\rho)=a\rho
\end{align}
for a given $a>0$.
Assume that $(\trho_0,\tu_0)$ and $(\rho_0,\vu_0)$ satisfy \eqref{initial condition on both systems} and \eqref{initial data being L2}. There is a positive constant $M$ depending on $T$, $\bar{C}_1,\bar{C}_2,\bar{C}_3,\bar{C}_4,\bar{\rho}_1,\bar{\rho}_2$, $\bar{\rho}$, $\mu$, $\lambda$, and $a$ but independent of $K$ such that if $(\trho,\tu)$ and $(\rho,\vu)$ are weak solutions of the systems (NSENBC) and (NSEDBC) respectively satisfying regularity assumptions \eqref{new spaces for weak sol 1}--\eqref{spaces for weak sol 5}, then
\begin{align}\label{bound on difference}
&\left(\intoxT|\vu-\tu|^2dxd\tau\right)^\frac{1}{2}+\sup_{0\le \tau\le T}\|(\rho-\trho)(\cdot,\tau)\|_{H^{-1}(\Omega)}\notag\\
&\le M\Big(\|\rho_0-\trho_0\|_{L^2(\Omega)}+\|\rho_0\vu_0-\trho_0\tu_0\|_{L^2(\Omega)}+K^\frac{1}{2}\Big).
\end{align}
\end{theorem}

\begin{remark}
As a corollary of Theorem~\ref{convergence weak thm}, if we take $(\trho_0,\tu_0)\equiv(\rho_0,\vu_0)$, then it implies the following convergence of weak solutions as $K\to0$:
\begin{align}\label{weak convergence in weak limit}
\lim_{K\to0}\Big(\|\vu-\tu\|_{L^2(\Omega\times(0,T))}+\sup_{0\le \tau\le T}\|(\rho-\trho)(\cdot,\tau)\|_{H^{-1}}\Big)=0.
\end{align}
\end{remark}

The Lagrangian maps serve two purposes in the weak-solution argument.
They first allow us to compare the transported density fields associated
with the two velocity fields. They are also used to construct
transported test functions that preserve the normal boundary constraint
while permitting comparison of the momentum equations. The resulting
boundary terms are treated separately using the adjoint Navier
condition and the boundary estimates for the no-slip solution.

The rest of the paper is organized as follows. In Section~\ref{prelim section}, we
introduce the notation and preliminary estimates used throughout the paper. In Section~\ref{lame operator section}, we study the Lam\'e operator with Navier boundary
conditions and establish the regularity estimates that are uniform with
respect to the slip length $K$. In Section~\ref{existence section}, we derive the
$K$-independent a priori estimates for the density and velocity and prove Theorem~\ref{local-in-time existence theorem}. In Section~\ref{convergence section}, we compare the smooth Navier and no-slip
solutions and prove the strong convergence result in Theorem~\ref{convergence smooth thm}. In
Section~\ref{weak solution section}, we introduce the Lagrangian flow maps associated with the weak
solutions and use transported test functions to estimate the density and
velocity differences. This yields the weak convergence result in
Theorem~\ref{convergence weak thm}.

\section{Preliminaries and notations}\label{prelim section}

We introduce the following notations and conventions:
\begin{itemize}
\item For $s\ge0$ and $p\ge1$, $W^{s,p}(\Omega)$ is the usual Sobolev space with norm $\|\cdot\|_{W^{s,p}(\Omega)}$ given by
\begin{align*}
\|\cdot\|_{W^{s,p}(\Omega)}(f):=\sum_{|\beta|\le s} \bigg(\intox|D^{\beta}_x f(x)|^pdx\bigg)^\frac{1}{p}.
\end{align*}
We write $H^s(\Omega):=W^{s,2}(\Omega)$. For simplicity, we also write $H^s=H^s(\Omega)$, $\|\cdot\|_{L^p}=\|\cdot\|_{L^p(\Omega)}$, $\|\cdot\|_{W^{s,p}}=\|\cdot\|_{W^{s,p}(\Omega)}$, etc. unless otherwise specified.
\item We adopt the usual notation for H\"older seminorms, namely for $v:\R^3\to \R^3$ and $\beta \in (0,1]$,
\hfill
$$\langle v\rangle^\beta = \sup_{{x_1,x_2\in
\R^3}\atop{x_1\not=x_2}}
{{|v(x_2) -v(x_1)|}\over{|x_2-x_1|^\beta}}\,;$$
and for $v:Q\subseteq\R^3 \times[0,\infty)\to \R^3$ and $\beta_1,\beta_2 \in (0,1]$,
\hfill
$$\langle v\rangle^{\beta_1,\beta_2}_{Q} = \sup_{{(x_1,t_1),(x_2,t_2)\in
Q}\atop{(x_1,t_1)\not=(x_2,t_2)}}
{{|v(x_2,t_2) - v(x_1,t_1)|}\over{|x_2-x_1|^{\beta_1} + |t_2-t_1|^{\beta_2}}}\,.$$
\item Regarding the constants used in this work, $C$ shall denote a positive and sufficiently large constant, whose value may change from line to line.
\end{itemize}

We recall the following standard facts (see, e.g., \cite{evans10,ziemer89}) which will be useful for later analysis.
\begin{itemize}
\item There is a constant $C = C(\Omega)$ such that, for all $\varphi\in H^2$,
\begin{align}\label{bound on L infty norm}
\|\varphi\|_{L^\infty}\le C\|\varphi\|_{H^2}.
\end{align}
\item For $p\in[2,6]$, there is a constant $C=C(p,\Omega)$ such that, for $\varphi\in H^1$,
\begin{align}\label{embedding from Lp to H1}
\|\varphi\|_{L^p}\le C\bigg(\|\varphi\|_{L^2}+\|\varphi\|_{L^2}^\frac{6-p}{2p}
\|\nabla\varphi\|_{L^2}^\frac{3p-6}{2p}\bigg).
\end{align}
\item Assume that $\partial\Omega$ is $C^1$, then there exists a constant $C = C(\Omega)$ such that, for all $\varphi\in W^{1,1}$,
\begin{align}\label{bound on bdry integral}
\int_{\partial\Omega}|\varphi(x)|^2dS_x\le C\intox\big(|\varphi(x)|^2+|\varphi(x)||\nabla\varphi(x)|\big)dx.
\end{align}
\end{itemize}

\section{The Lam\'{e} operator}\label{lame operator section}

In this section we introduce an important differential operator which is related to the system \eqref{Navier Stokes}. For $\mu$ and $\lambda$ satisfying \eqref{conditions on viscosity}, let $\cL$ denote the Lam\'{e} operator given by
\begin{align}\label{def of lame}
(\cL \vu)^j=\mu\Delta u^j + \lambda(\divv\vu)_{x_j}.
\end{align}
With respect to the Lam\'{e} operator $\cL$, we define the system
\begin{align}\label{Lame system}
\left\{ \begin{array}{l}
\cL \vu = -g \mbox{\,\,\,in $\Omega$,}\\
\mbox{$\vu\cdot \No(x)=0$ and $(u^1(x),u^2(x),u^3(x))=K\mu(u^1_{x_3}(x),u^2_{x_3}(x),0)$ for $x\in\partial\Omega$,}
\end{array}\right.
\end{align}
where $\vu:\Omega\to\R^3$ is the unknown function and $g:\Omega\to\R^3$ is given. We further define a Hilbert space $V$ with the $H^1$ inner product given by
\begin{align}\label{def of space V}
V=\{\vw\in (H^1(\Omega))^3:\vw\cdot\No=0\mbox{ on $\partial\Omega$}\}.
\end{align}

We seek for $\vu\in V$ such that for all $\vvv\in V$, the following identity holds
\begin{align}\label{weak form of Lame operator problem}
\intox g\cdot\vvv dx = \intox\big(\mu\nabla\vu\cdot\nabla\vvv+\lambda\divv\vu\divv\vvv\big)dx
+\alpha\intoxb\vu\cdot\vvv dS_{x}.
\end{align}
The equation \eqref{weak form of Lame operator problem} can be viewed as a {\it weak form} of \eqref{Lame system} with $V$ being the solution space.

We point out that, by taking $g=\rho\dot{\vu}+\nabla P(\rho)$ where $\dot{\vu}=\vu_t+\vu\cdot \nabla \vu$, the system \eqref{Lame system} gives the Navier--Stokes equations \eqref{Navier Stokes}. The following lemma gives an estimate on $\vu$ in terms of $g$ which is crucial for obtaining bounds on solutions to the system (NSENBC) later.

\begin{lemma}\label{estimate on lame operator lemma}
Let $V$ be defined by \eqref{def of space V}. Assume that $\mu>0$ and $\mu+3\lambda>0$. Let $m\in\{0,1\}$ and
$g\in H^{m}(\Omega;\mathbb{R}^{3})$. Suppose that
$\vu\in V$ satisfies
$$
\mu\int_{\Omega}\nabla\vu:\nabla\vvv\,dx
+\lambda\int_{\Omega}\divv\vu\,\divv\vvv\,dx
+K^{-1}\int_{\partial\Omega}\vu_{\tau}\cdot\vvv_{\tau}\,dS
=
\int_{\Omega}g\cdot\vvv\,dx
\label{3.4}
$$
for every $\vvv\in V$, where $\vvv_{\tau}$ and $\vu_{\tau}$ are the tangential components of $\vvv$ and $\vu$ respectively given by
$$
\vvv_{\tau}:=\vvv-(\vvv\cdot \No)\No,
\qquad
\vu_{\tau}:=\vu-(\vu\cdot \No)\No.
$$
Then we have $\vu\in H^{m+2}(\Omega)$ and
$$
\|\vu\|_{H^{m+2}(\Omega)}
\leq
C\left(
\|\vu\|_{L^{2}(\Omega)}
+\|g\|_{H^{m}(\Omega)}
\right),
\label{elliptic estimate on lame operator}
$$
where $C$ depends only on $\mu$, $\lambda$, and $\Omega$, and is
independent of $K$.
\end{lemma}

\begin{proof}
We divide the proof into several steps as follows.

\textbf{Step 1. The basic $H^{1}$ estimate:} Taking $\vvv=\vu$ in \eqref{weak form of Lame operator problem}, we obtain
\begin{equation}
\mu\|\nabla\vu\|_{L^{2}(\Omega)}^{2}
+\lambda\|\divv\vu\|_{L^{2}(\Omega)}^{2}
+K^{-1}\|\vu_{\tau}\|_{L^{2}(\partial\Omega)}^{2}
=
\int_{\Omega}g\cdot\vu\,dx.
\label{3.6}
\end{equation}
The assumptions $\mu>0$ and $\mu+3\lambda>0$ imply that there exists
$c_{\mu,\lambda}>0$ such that
$$
\mu |A|^{2}
+\lambda|\operatorname{tr}A|^{2}
\geq
c_{\mu,\lambda}|A|^{2}
\qquad
\text{for every }A\in\mathbb{R}^{3\times3}.
$$
where $\operatorname{tr} A$ is the trace of $A$. Consequently, we have
$$
\mu\|\nabla\vu\|_{L^{2}(\Omega)}^{2}
+\lambda\|\divv\vu\|_{L^{2}(\Omega)}^{2}
\geq
c_{\mu,\lambda}\|\nabla\vu\|_{L^{2}(\Omega)}^{2}.
$$
Since the boundary term in \eqref{3.6} is nonnegative, it can be
discarded in the estimate. Hence
$$
c_{\mu,\lambda}\|\nabla\vu\|_{L^{2}(\Omega)}^{2}
\leq
\|g\|_{L^{2}(\Omega)}
\|\vu\|_{L^{2}(\Omega)}.
$$
By Young's inequality,
$$
\|\nabla\vu\|_{L^{2}(\Omega)}^{2}
\leq
C\left(
\|g\|_{L^{2}(\Omega)}^{2} +
\|\vu\|_{L^{2}(\Omega)}^{2}\right).
$$
Therefore we have
\begin{equation}
\|\vu\|_{H^{1}(\Omega)}
+
K^{-1/2}\|\vu_{\tau}\|_{L^{2}(\partial\Omega)}
\leq
C\left(
\|\vu\|_{L^{2}(\Omega)}
+\|g\|_{L^{2}(\Omega)}
\right),
\label{3.7}
\end{equation}
where $C$ is independent of $K$.

\textbf{Step 2. Tangential difference quotients:} For $i\in\{1,2\}$ and $h\neq0$, we define the tangential difference quotient by
$$
\delta_{h}^{i}f(x)
:=
\frac{f(x+he_{i})-f(x)}{h},
$$
where $\{e_1,e_2,e_3\}$ is the standard unit vectors in $\R^3$. Since the half-space is invariant under translations in the
$x_{1}$- and $x_{2}$-directions, $\delta_{h}^{i}\vu$ is defined
on $\Omega$ and satisfies
$$
\delta_{h}^{i}\vu\cdot n=0
\quad\text{on }\partial\Omega.
$$
Moreover, the Navier boundary condition is preserved under tangential difference quotients. We use
$$
\vvv=\delta_{-h}^{i}\delta_{h}^{i}\vu
$$
as a test function in \eqref{weak form of Lame operator problem}. The discrete integration-by-parts
identity
$$
\int_{\Omega}f\cdot
\delta_{-h}^{i}\delta_{h}^{i}f\,dx
=
-\int_{\Omega}
|\delta_{h}^{i}f|^{2}\,dx
$$
and its corresponding versions for $\nabla\vu$ and $\divv\vu$ yield
$$
\begin{aligned}
\mu\|\nabla\delta_{h}^{i}\vu\|_{L^{2}(\Omega)}^{2}
+\lambda\|\divv\delta_{h}^{i}\vu\|_{L^{2}(\Omega)}^{2}
+K^{-1}
\|(\delta_{h}^{i}\vu)_{\tau}\|_{L^{2}(\partial\Omega)}^{2}
\leq
\left|
\int_{\Omega}
\delta_{h}^{i}g\cdot
\delta_{h}^{i}\vu\,dx
\right|.
\end{aligned}
$$
Using the coercivity of the Lam\'e operator and Cauchy--Schwarz inequality, we obtain
$$
\|\nabla\delta_{h}^{i}\vu\|_{L^{2}(\Omega)}^{2}
\leq
C\left(
\|\delta_{h}^{i}g\|_{L^{2}(\Omega)}^{2}
+
\|\delta_{h}^{i}\vu\|_{L^{2}(\Omega)}^{2}
\right).
$$
Thus,
\begin{align}
\|\delta_{h}^{i}\vu\|_{H^{1}(\Omega)}
\leq
C\left(
\|\delta_{h}^{i}\vu\|_{L^{2}(\Omega)}
+
\|\delta_{h}^{i}g\|_{L^{2}(\Omega)}
\right),
\label{3.8}
\end{align}
where $C$ is independent of $h$ and $K$.

Since $\vu\in H^{1}(\Omega)$ and
$g\in H^{1}(\Omega)$ when $m=1$, the right-hand side is uniformly
bounded as $h\to0$. Passing to the limit gives
\begin{align}
\|\nabla_{\tau}\vu\|_{H^{1}(\Omega)}
\leq
C\left(
\|\vu\|_{H^{1}(\Omega)} +
\|\nabla_{\tau}g\|_{L^{2}(\Omega)}
\right).
\label{3.9}
\end{align}

For the case $m=0$, since the difference-quotient estimate used for the case $m=1$ involves
$\delta_h^{i} g$ and therefore cannot be applied directly when
$g\in L^2(\Omega)$. The following half-space regularity estimate is a standard
$K$-uniform Agmon--Douglis--Nirenberg estimate for the Lam\'e
operator with mixed Dirichlet--Navier boundary conditions:
\begin{equation}
\|\nabla_{\tau}^{2}\vu\|_{L^{2}(\Omega)}
+
\|\nabla_{\tau}\partial_{x_3}\vu\|_{L^{2}(\Omega)}
\leq
C\left(
\|\vu\|_{H^{1}(\Omega)}
+\|g\|_{L^{2}(\Omega)}
\right).
\label{3.7'}
\end{equation}
Here the constant is independent of $K$, since the boundary
integral
$
K^{-1}\|\vu_{\tau}\|_{L^2(\partial\Omega)}^2
$
is nonnegative. This is the flat-boundary analogue of the elliptic
regularity argument used by Hoff \cite{hoff05}, where the Lam\'e system is reduced to
scalar Poisson equations with Robin-type boundary conditions.

\textbf{Step 3. Recovery of the second normal derivatives:} The equation corresponding to \eqref{weak form of Lame operator problem} is
\begin{align}
-\mu\Delta\vu-\lambda\nabla\divv\vu=g
\quad\text{in }\Omega.
\label{3.10}
\end{align}
Writing this equation componentwise, for $j=1,2$ we have
$$
-\mu\partial_{x_3 x_3}u^{j}
=
g^{j}
+\mu\Delta_{\tau}u^{j}
+\lambda\partial_{x_j}\divv\vu,
$$
where
$$
\Delta_{\tau}
=
\partial_{x_1 x_1}+\partial_{x_2 x_2}.
$$
For the third component,
$$
-(\mu+\lambda)\partial_{x_3 x_3}u^{3}
=
g^{3}
+\mu\Delta_{\tau}u^{3}
+\lambda\partial_{x_3}
\left(
\partial_{x_1}u^{1}
+\partial_{x_2}u^{2}
\right).
$$
Since $\mu+\lambda>0$ follows from $\mu+3\lambda>0$, these relations imply
$$
\|\partial_{x_3 x_3}\vu\|_{L^{2}(\Omega)}
\leq
C\left(
\|g\|_{L^{2}(\Omega)} +
\|\nabla_{\tau}^{2}\vu\|_{L^{2}(\Omega)} +
\|\nabla_{\tau}\partial_{x_3}\vu\|_{L^{2}(\Omega)}\right).
$$
Using \eqref{3.7} and \eqref{3.9}, we obtain
\begin{align}
\|\vu\|_{H^{2}(\Omega)}
\leq
C\left(
\|\vu\|_{L^{2}(\Omega)}
+\|g\|_{L^{2}(\Omega)}
\right).
\label{3.11}
\end{align}

\textbf{Step 4. The $H^{3}$ estimate:} Applying the tangential difference-quotient
argument of {\bf Step 2} to $\partial_{x_i}\vu$ for $i\in\{1,2\}$, it gives
\begin{align}
\|\nabla_{\tau}\vu\|_{H^{2}(\Omega)}
\leq
C\left(
\|\vu\|_{H^{1}(\Omega)}
+\|g\|_{H^{1}(\Omega)}
\right).
\label{3.12}
\end{align}
Consequently, all third-order derivatives of $\vu$ containing at least
one tangential derivative are controlled by
$$
C\left(
|\vu\|_{H^{1}(\Omega)}
+\|g\|_{H^{1}(\Omega)}
\right).
$$
To control the remaining third-order normal derivative, we differentiate
the third component of \eqref{3.10} with respect to $x_{3}$, which gives
$$
\begin{aligned}
-(\mu+\lambda)\partial_{x_3x_3 x_3}u^{3}
=
\partial_{x_3}g^{3}
+\mu\partial_{x_3}\Delta_{\tau}u^{3}+\lambda\partial_{x_3 x_3}
\left(
\partial_{x_1}u^{1}
+\partial_{x_2}u^{2}
\right).
\end{aligned}
$$
Therefore,
$$
\|\partial_{x_3x_3 x_3}u^{3}\|_{L^{2}(\Omega)}
\leq
C
\left(\|\nabla g\|_{L^{2}(\Omega)} +
\|\nabla_{\tau}^{2}\partial_{x_3}u^{3}\|_{L^{2}(\Omega)} +
\|\nabla_{\tau}\partial_{x_3 x_3}u^{3}\|_{L^{2}(\Omega)}\right).
$$
The terms on the right-hand side are controlled by
\eqref{3.11} and \eqref{3.12}. For other terms $\partial_{x_3 x_3 x_3}u^{1}$ and $\partial_{x_3 x_3 x_3}u^{2}$, theses can be recovered by differentiating the first two component equations in $x_3$, which gives
$$
-\mu\partial_{x_3 x_3 x_3}u^j
\partial_{x_3}g^j +\mu\partial_{x_3}\Delta_\tau u^j +\lambda\partial_{x_3}\partial_{x_j}\divv u, \qquad j=1,2.
$$
The right-hand side can be controlled using the tangential $H^2$-estimate and the estimate for $\partial_{x_3 x_3 x_3}u^{3}$. Hence we have
$$
\|\vu\|_{H^{3}(\Omega)}
\leq
C\left(
\|\vu\|_{H^{1}(\Omega)}
+\|g\|_{H^{1}(\Omega)}
\right).
$$
Again using \eqref{3.7}, it gives
\begin{align}
\|\vu\|_{H^{3}(\Omega)}
\leq
C\left(
\|\vu\|_{L^{2}(\Omega)}
+\|g\|_{H^{1}(\Omega)}
\right).
\label{3.13}
\end{align}

Combining \eqref{3.11} and \eqref{3.13}, we conclude that
$$
\|\vu\|_{H^{m+2}(\Omega)}
\leq
C\left(
\|\vu\|_{L^{2}(\Omega)}
+\|g\|_{H^{m}(\Omega)}
\right),
\qquad m\in\{0,1\}.
$$
The constant $C$ is independent of $K$, because the only
$K$-dependent boundary contribution is
$$
K^{-1}\|\vu_{\tau}\|_{L^{2}(\partial\Omega)}^{2},
$$
which has the correct sign and is retained on the left-hand side. This proves \eqref{elliptic estimate on lame operator}.
\end{proof}

\begin{remark}
We observe that $K$ (or equivalently $\alpha$) has the {\it correct} sign, which allows us to discard those corresponding boundary terms appeared in the above analysis. This is also important for proving the $K$-independence for the constant $C$ as in \eqref{elliptic estimate on lame operator}. In particular, if $(\rho,\vu)$ is a smooth solution to (NSENBC), then there is a constant $C = C(\Omega,\mu,\lambda)$ independent of $K$ such that for $m=0$ or 1,
\begin{align}\label{elliptic estimate on NSENBC}
\|\vu\|_{H^{m+2}}\le C(\|\vu\|_{H^m}+\|\rho\dot{\vu}\|_{H^m}+\|\nabla P(\rho)\|_{H^m}).
\end{align}
The estimate \eqref{elliptic estimate on NSENBC} will become useful in later sections.
\end{remark}

\section{Existence of local-in-time smooth solutions: Proof of Theorem~\ref{local-in-time existence theorem}}\label{existence section}

In this section we give the proof of Theorem~\ref{local-in-time existence theorem}. We only focus on the system (NSENBC) since the local-in-time existence of (NSEDBC) can be proved in the same way as that in \cite{tani77}. Throughout this section, $M_0$ will be fixed as in the statement of Theorem~\ref{local-in-time existence theorem}.

We first give the following definitions of function spaces that will be used for later analysis. More details can be found in \cite{hoff12}.

\begin{definition}
For $M\ge M_0$ and $T>0$, $\seta$ is the set of maps $(\rho-\brho):[0,T]\to H^2(\Omega)$ such that $\rho(\cdot,0)=\rho_0$, $(\rho-\brho)\in C([0,T];H^1(\Omega))$, $\rho_t\in C([0,T];L^2(\Omega))$, and
\begin{align}\label{bound on rho def}
\sup_{0\le t\le T}\big(\|\rho(\cdot,t)\|_{L^\infty},\|\rho^{-1}(\cdot,t)\|_{L^\infty},
\|\rho(\cdot,t)-\brho\|_{H^2}\big)\le M.
\end{align}
And for $N>0$, $\setb$ is the set of maps $\vu:[0,T]\to H^3(\Omega)$ such that $\vu\cdot \No(x)=0$, $\vu(\cdot,0)=\vu_0$, $\vu\in C([0,T];H^2(\Omega))$, $\vu_t\in C([0,T];L^2(\Omega))$ and
\begin{align}\label{bound on u def}
\sup_{0\le t\le T}\big(\|\vu(\cdot,t)\|^2_{H^3},\|\vu_t(\cdot,t)\|^2_{H^1}\big)
+\int_0^T\|\vu_t(\cdot,t)\|^2_{H^2}\le N.
\end{align}
\end{definition}
We first obtain the following lemma which shows that given $\vu\in\setb$, there is a corresponding solution $\rho$ to the mass equation \eqref{Navier Stokes}$_1$ with initial data $\rho_0$.

\begin{lemma}\label{estimate on rho lemma}
Given $N>0$ and $M'_0>M_0$, there is $T_1=T_1(M_0,M'_0,N)>0$ such that if $\vu\in\setb$ for some $T>0$, then there is a unique $\rho\in\tilde\sum_1^{M'_0,\min\{T_1,T\}}$ such that the pair $(\rho,\vu)$ satisfies the equation \eqref{Navier Stokes}$_1$ such that
\begin{align}\label{bound on nabla t rho thm}
\int_0^{\min\{T_1,T\}}|\nabla\rho_t|^2dxdt\le 1
\end{align}
and
\begin{align}\label{bound on rho t thm}
\sup_{0\le t\le\min\{T_1,T\}}\|\rho_t(\cdot,t)\|_{L^2}\le C_1M_0^2
\end{align}
for a constant $C_1>0$ which depends only on $\Omega$.  Also, there is a constant $C=C(M_0,M'_0,N)$ such that
\begin{align}\label{holder continuity of rho thm}
\langle\rho\rangle^{\frac{1}{2},\frac{1}{2}}_{\bar{\Omega}\times[0,T]}\le C
\end{align}
and
\begin{align}\label{continuity of rho and rhot in time thm}
\|\rho(\cdot,t_1)-\rho(\cdot,t_2)\|_{H^1},\|\rho_t(\cdot,t_1)-\rho_t(\cdot,t_2)\|_{L^2}\le C|t_2-t_1|
\end{align}
for all $t_1$, $t_2\in[0,\min\{T_1,T\}]$.
\end{lemma}

\begin{proof}
The proof is similar to the one given for \cite[Theorem~3.2]{hoff12}. The key is to consider the characteristics $X(y,t;t_0)$ for the first-order equation
\begin{align*}
\rho_t + \nabla\rho\cdot\vu + \rho\divv \vu = 0,
\end{align*}
which are given by
\begin{align*}
\frac{dX}{dt} &= \vu (X,t),\\
X(y,t_0;t_0) &= y,
\end{align*}
for all $(y,t_0,t_0)\in\Omega\times[0,T]\times[0,T]$. Hence $\rho$ can be represented by
\begin{align}\label{representation of rho using X}
\rho(x,t) = \rho_0(X(x,0;t))\exp\Big(-\int_0^t\divv \vu(X(x,s;t),s)ds\Big).
\end{align}
The bounds \eqref{bound on nabla t rho thm}--\eqref{bound on rho t thm} follow by performing standard energy estimates on $\|\rho-\brho\|_{H^2}$ with the prescribed bounds \eqref{bound on u def} on $\vu$. For the H\"{o}lder continuity in time, we notice that for $y\in\Omega$, we have
\begin{align*}
\frac{d}{dt}\rho(X(y,t;t_0),t)=-\rho\divv \vu.
\end{align*}
Using the $H^2$-bounds on $\rho$ and $\vu$, one can show that the maps $t\mapsto\rho(X(y,t;t_0)$ are Lipschitz continuous. Hence for sufficiently small $T_1>0$,
\begin{align*}
|\rho(x,t_2)-\rho(x,t_1)|&\le|\rho(x,t_2)-\rho(X(x,t_2;t_1),t_2)|+|\rho(X(x,t_2;t_1),t_2)-\rho(x,t_1)|\\
&\le C|X(x,t_1;t_1)-X(x,t_2;t_1)|^\frac{1}{2}+C|t_2-t_1|\\
&\le C|t_2-t_1|^\frac{1}{2},
\end{align*}
whenever $t_1,t_2\in(0,T_1]$.
\end{proof}

Next, we reverse the role of $\rho$ and $\vu$ and obtain estimates on $\vu$ determined by a given density $\rho\in\seta$. The results are summarized in the following lemma:

\begin{lemma}\label{estimate on u lemma}
Given $M\ge M_0$, there is a positive time $T_2 = T_2(M_0,M)$ and a constant $N = N(M_0,M)$ independent of $K$ such that if $\rho\in\seta$ and satisfies \eqref{bound on nabla t rho thm}--\eqref{continuity of rho and rhot in time thm} for some $T>0$, then there is a unique $\vu\in\tilde\sum_2^{N,\min\{T_2,T\}}$ with $\vu(\cdot,t)=\vu_0$ satisfying
\begin{gather}\label{holder estimates on u in space time}
\langle\vu\rangle^{\frac{1}{2},\frac{1}{4}}_{\bar{\Omega}\times[0,T]}\le N,
\\ \label{continuity in time for u and ut}
\|\vu(\cdot,t_2)-\vu(\cdot,t_1)\|_{H^2},\|\vu_t(\cdot,t_2)-\vu_t(\cdot,t_1)\|_{L^2}\le N|t_2-t_1|^\frac{1}{2},
\end{gather}
for $t_1$, $t_2\in[0,\min\{T_2,T\}]$.
\end{lemma}

\begin{proof}
Lemma~\ref{estimate on u lemma} can be proved in several steps given as follows:

{\bf Step 1. Linearising \eqref{Navier Stokes}$_2$:} Fix an admissible density $\rho$ and a prescribed velocity $\vvv$. We consider the linearised parabolic equation
\begin{align}\label{linearised problem lem 4.2}
\rho \vu_t+\rho \vvv\cdot\nabla \vu+\nabla P(\rho)
=\mu\Delta \vu+\lambda\nabla\divv \vu
\end{align}
which is equipped with initial conditions \eqref{initial condition}--\eqref{Navier condition}. The existence of unique solutions $\vu$ to \eqref{linearised problem lem 4.2} be proved by the same argument given in \cite{hoff05}, and $\vu$ further satisfies
\begin{align*}
\|\vu(\cdot,t)\|_{H^2}\le C+C\int_0^t\|\vvv(\cdot,s)\|_{H^2}ds,
\end{align*}
where $C$ is a constant which depends on $\rho$ and $P$.

{\bf Step 2. A priori estimates:}
Next we obtain the bound \eqref{bound on u def} on $\vu$ and the proof consists of four separate energy-type estimates which will be carried out subsequently. Most of the details are reminiscent of those given in \cite{hoff12} except we have to ensure that the constant $N$ is independent of $K$. Throughout this proof, $N = N(M_0,M)$ will denote a generic positive constant as described in the statement of Lemma~\ref{estimate on u lemma}.

{\bf (i) $L^2$ bound.} We multiply \eqref{Navier Stokes}$_2$ by $u^j$, apply the boundary condition \eqref{Navier condition} and sum over $j$ to obtain
\begin{align*}
&\frac{d}{dt}\intox\frac{1}{2}\rho|\vu|^2dx+\mu\intox|\nabla\vu|^2dx+K^{-1}\intoxb|\vu|^2dS_x\\
&\le N\intox\big(|\nabla P(\rho)||\vu|+|\rho_t+\divv(\rho\vu)||\vu|^2\big)dx.
\end{align*}
Since $K>0$, the boundary integral can be discarded from the left side, and hence we apply the bound \eqref{bound on rho def} on $\rho$ to conclude
\begin{align}\label{L2 bound on u}
\frac{d}{dt}\intox|\vu|^2dx+\intox|\nabla\vu|^2dx\le N\big(1+\|\vu(\cdot,t)\|^3_{H^1}\big).
\end{align}

{\bf (ii) $H^1$ bound.} Next we multiply \eqref{Navier Stokes}$_2$ by $u^j_t$, apply the boundary condition \eqref{Navier condition} and sum over $j$ to obtain
\begin{align*}
\frac{d}{dt}\bigg(\frac{1}{2}\intox(\mu|\nabla\vu|^2+\lambda(\divv\vu)^2dx
+\frac{1}{2K}\intoxb|\vu|^2dS_x\bigg)+N^{-1}\intox|\vu_t|^2dx
\le N\bigg(1+\intox|\vu|^2|\nabla\vu|^2dx\bigg).
\end{align*}
To bound the term $\dis\intox|\vu|^2|\nabla\vu|^2dx$, using H\"{o}lder's inequality and the embedding \eqref{embedding from Lp to H1}, we have
\begin{align}\label{bound on |u|2|Du|2}
\intox|\vu|^2|\nabla\vu|^2dx\le \left(\intox|\vu|^6dx\right)^\frac{1}{3}\left(\intox|\nabla\vu|^3dx\right)^\frac{2}{3}
\le N\big(\|\vu\|^4_{H^1}+\|\vu\|^3_{H^1}\|\cL\vu\|_{L^2}\big),
\end{align}
where $\cL$ is the Lam\'{e} operator defined in \eqref{def of lame} and we have used the fact that $\cL$ is elliptic. Using the estimate \eqref{elliptic estimate on NSENBC} and the bound \eqref{bound on rho def} on $\rho$, one gets that
\begin{align}\label{L2 bound on Lu}
\intox|\cL\vu|^2dx&\le N\left(\intox|\rho\vu_t|^2dx+\intox|\vu|^2|\nabla\vu|^2dx+\intox|\nabla P(\rho)|^2dx\right)\notag\\
&\le N\left(\intox|\vu_t|^2dx+\intox|\vu|^2|\nabla\vu|^2dx\right)+N.
\end{align}
Substituting \eqref{L2 bound on Lu} into \eqref{bound on |u|2|Du|2} yields
\begin{align*}
\intox|\vu|^2|\nabla\vu|^2dx\le N\left(\|\vu|^4_{H^1}+\|\vu\|^3_{H^1}\left(\intox|\vu_t|^2dx+\intox|\vu|^2|\nabla\vu|^2dx+1\right)^\frac{1}{2}\right),
\end{align*}
which further gives
\begin{align*}
\intox|\vu|^2|\nabla\vu|^2dx\le N\big(1+\|\vu\|^6_{H^1}+\|\vu\|^3_{H^1}\|\vu_t\|_{L^2}\big).
\end{align*}
Hence, there is a positive time $T_2=T_2(M_0,M)$ independent of $K$ such that
\begin{align}\label{H1 bound on u}
\sup_{0\le t\le\min\{T_2,T\}}\|\vu(\cdot,t)\|^2_{H^1}+\int_0^{\min\{T_2,T\}}\intox|\vu_t|^2dxdt\le N.
\end{align}

{\bf (iii) $H^2$ bound.} We multiply \eqref{Navier Stokes}$_2$ by $u^j_{tt}$, apply the boundary condition \eqref{Navier condition}, and sum over $j$ to obtain
\begin{align*}
\frac{d}{dt}\intox\frac{1}{2}\rho|\vu_t|^2dx+N^{-1}\intox|\nabla\vu_t|^2dx
+K^{-1}\intoxb|\vu_t|^2dS_x\le N\big(1+\|\vu_t\|^4_{L^2}\big).
\end{align*}
Since $K>0$, the boundary integral appeared above can be discarded. It follows that for a new time $T_2 = T_2(M_0,M)$,
\begin{align}\label{L2 bound on ut}
\sup_{0\le t\le \min\{T_2,T\}}\intox|\vu_t|^2dx+\int_0^{\min\{T_2,T\}}\intox|\nabla\vu_t|^2dxdt\le N.
\end{align}
Furthermore, by applying the bounds \eqref{H1 bound on u} and \eqref{L2 bound on ut} on \eqref{L2 bound on Lu}, we can see that $\|\cL\vu\|_{L^2}$ can be bounded by $N$, which implies that
\begin{align}\label{H2 bound on u}
\sup_{0\le t\le\min\{T_2,T\}}\|\vu(\cdot,t)\|_{H^2}\le N.
\end{align}

{\bf (iv) $H^3$ bound.} Finally, we multiply \eqref{Navier Stokes}$_2$ by $\cL u^j_t$ and sum over $j$ to obtain
\begin{align}\label{H3 bound step 1}
\intox\{\rho(\vu_t+\vu\cdot\nabla\vu)+\nabla P(\rho)-\cL\vu\}_t\cdot\cL\vu_t dx=0.
\end{align}
We compute the term $\dis\intox\rho\vu_{tt}\cdot\Delta\vu_t$ as appeared in \eqref{H3 bound step 1}. Using the boundary condition \eqref{Navier condition},
\begin{align*}
\intox\rho\vu_{tt}\cdot\Delta\vu_tdx=-\frac{d}{dt}\intox\frac{1}{2}\rho|\nabla\vu_t|^2dx
+\intox\frac{1}{2}\rho_t|\nabla\vu_t|^2dx
-K^{-1}\frac{d}{dt}\intoxb\rho|\vu_t|^2dS_x-\intoxb\frac{1}{2}\rho_t\vu_t\cdot\nabla\vu_tdS_x,
\end{align*}
and by applying \eqref{bound on bdry integral}, the boundary integral $\dis-\intoxb\frac{1}{2}\rho_t\vu_t\cdot\nabla\vu_tdS_x$ can be bounded by
\begin{align*}
\left|\intoxb\frac{1}{2}\rho_t\vu_t\cdot\nabla\vu_tdS_x\right|
&\le\frac{1}{2}\left(\intoxb|\rho_t\vu_t|^2dS_x\right)^\frac{1}{2}
\left(\intoxb|\nabla\vu_t|^2dS_x\right)^\frac{1}{2}\\
&\le N\|\vu_t\|_{H^2}\left\{\bigg(\intox|\rho_t\vu_t|^2dx\bigg)^\frac{1}{2}
+\bigg(\intox|\rho_t\vu_t|^2dx\bigg)^\frac{1}{4}
\bigg(\intox|\nabla\rho_t|^2|\vu_t|^2dx\bigg)^\frac{1}{4}\right\}\\
&\quad+N\|\vu_t\|_{H^2}\left\{\bigg(\intox|\rho_t\vu_t|^2dx\bigg)^\frac{1}{4}
\bigg(\intox|\rho_t|^2|\nabla\vu_t|^2dx\bigg)^\frac{1}{4}\right\}.
\end{align*}
Hence, we obtain that
\begin{align}\label{H3 bound step 2}
\intox\rho\vu_{tt}\cdot\Delta\vu_tdx&\le-\frac{d}{dt}
\left(\intox\frac{1}{2}\rho|\nabla\vu_t|^2dx
+K^{-1}\intoxb\rho|\vu_t|^2dS_x\right)\notag\\
&\quad+N\|\vu_t\|_{H^2}\left\{\bigg(\intox|\rho_t\vu_t|^2dx\bigg)^\frac{1}{2}
+\bigg(\intox|\rho_t\vu_t|^2dx\bigg)^\frac{1}{4}
\bigg(\intox|\nabla\rho_t|^2|\vu_t|^2dx\bigg)^\frac{1}{4}\right\}\notag\\
&\quad+N\|\vu_t\|_{H^2}\left\{\bigg(\intox|\rho_t\vu_t|^2dx\bigg)^\frac{1}{4}
\bigg(\intox|\rho_t|^2|\nabla\vu_t|^2dx\bigg)^\frac{1}{4}\right\}.
\end{align}
Notice that $\dis K^{-1}\intoxb\rho|\vu_t|^2dS_x\ge0$, it can be dropped off from \eqref{H3 bound step 2} after we integrate over time. The other terms in \eqref{H3 bound step 1} can be treated in a similar way as we did before so that by applying \eqref{H3 bound step 2} on \eqref{H3 bound step 1}, integrating over time, using the bounds available for $\rho\in\seta$ and performing a long but straightforward sequence of estimates, we conclude that
\begin{align}\label{H3 bound step 3}
\sup_{0\le t\le\min\{T_2,T\}}\intox|\nabla\vu_t|^2dx+\int_0^{\min\{T_2,T\}}\intox|\cL\vu_t|^2dxdt\le N.
\end{align}
Furthermore, by the estimate \eqref{elliptic estimate on NSENBC}, one has that
\begin{align}\label{control on H3 norm of u}
\|\vu\|_{H^{3}}\le N\big(\|\vu\|_{H^1}+\|\rho\dot{\vu}\|_{H^1}+\|\nabla P(\rho)\|_{H^1}\big).
\end{align}
Thus, using the bound \eqref{bound on rho def} on $\rho$ and the bounds \eqref{L2 bound on ut}, \eqref{H2 bound on u} and \eqref{H3 bound step 3} on $\vu$ and $\vu_t$, we obtain from \eqref{control on H3 norm of u} that
\begin{align*}
\sup_{0\le t\le T}\|\vu(\cdot,t)\|^2_{H^3}\le N.
\end{align*}
This combined with \eqref{H1 bound on u}, \eqref{L2 bound on ut}, and \eqref{H3 bound step 3} leads to
\begin{align}\label{H3 bound on u}
\sup_{0\le t\le T}\big(\|\vu(\cdot,t)\|^2_{H^3},\|\vu_t(\cdot,t)\|^2_{H^1}\big)
+\int_0^{\min\{T_2,T\}}\|\vu_t(\cdot,t)\|^2_{H^2}\le N,
\end{align}
which implies \eqref{bound on u def}.

{\bf Step 3. Recovering the nonlinear equation \eqref{Navier Stokes}$_2$:} Given $\vvv^{1}$ and $\vvv^{2}$, suppose $\vu^{1}$ and $\vu^{2}$ are the corresponding solution to the linearised equation \eqref{linearised problem lem 4.2} by taking $\vvv=\vvv^{1}$ and $\vvv=\vvv^{2}$ respectively. Then the difference $\vu^{1}-\vu^{2}$ satisfies
\begin{align*}
\rho(\vu^{1}-\vu^{2})_{t} + \rho\vvv^{1}\cdot\nabla (\vu^{1}-\vu^{2}) - \mu\Delta(\vu^{1}-\vu^{2}) - \lambda\nabla\divv (\vu^{1}-\vu^{2}) = -\rho(\vvv^{1}-\vvv^{2})\cdot\nabla\vu^{2}.
\end{align*}
By performing energy estimates and using the a priori bounds given in {\bf Step 2}, we obtain
\begin{align*}
\|\vu^{1} - \vu^{2}\|_{L^2} \le \tilde{C} T^{\theta}\|\vvv^{1} - \vvv^{2}\|_{L^2},
\end{align*}
for some $\tilde C, \theta>0$. Hence for $T$ being sufficiently small, the map $\vvv\mapsto\vu$ is a contraction and its fixed point is the required solution of the nonlinear momentum equation for the prescribed $\rho$ as claimed.
\end{proof}

\begin{proof}[Proof of Theorem~\ref{local-in-time existence theorem}]
We prove the assertion for the system (NSENBC), and the system (NSEDBC) can be treated analogously.

Let $N$ be the constant provided by Lemma~\ref{estimate on rho lemma} with
$M=M_0'$. Let $T_1$ and $T_2$ be the existence times given by
Lemma~\ref{estimate on rho lemma} and Lemma~\ref{estimate on u lemma} respectively, and
set $T=\min\{T_1,T_2\}$. By construction, $T$ and $N$ are independent of $K$.

Let $\mathcal B_N$ denote the closed bounded subset of
$\setb$ consisting of velocity fields with
initial data $u_0$ and satisfying the normal boundary condition. For
$u\in\mathcal B_N$, let $\rho=\mathcal R(u)$ be the unique solution of the continuity equation given by
Lemma~\ref{estimate on rho lemma}. Then $\rho\in\seta$ and $\rho$ satisfies \eqref{bound on nabla t rho thm}--\eqref{continuity of rho and rhot in time thm}.

We next define
$$
\mathcal T(u)=\mathcal V(\mathcal R(u)),
$$
where $\mathcal V(\rho)$ denotes the unique $\vu$ claimed by
Lemma~\ref{estimate on u lemma} for the prescribed density $\rho$. By Lemma~\ref{estimate on u lemma}, $
\mathcal T(u)\in\setb$ and therefore $\mathcal T$ maps $\mathcal B_N$ into itself.

The stability estimates in Lemma~\ref{estimate on rho lemma} and
Lemma~\ref{estimate on u lemma} imply that $\mathcal T$ is continuous in a
lower-order topology, say $C([0,T];L^2(\Omega))$. Moreover, the uniform estimates \eqref{bound on u def}, \eqref{holder estimates on u in space time} and \eqref{continuity in time for u and ut}, together with the compact
embedding on bounded subsets of $\Omega$, imply that
$\mathcal T(\mathcal B_N)$ is relatively compact in this topology.
The corresponding global compactness follows from the uniform
tightness estimates on the half-space. By Schauder's fixed-point
theorem, it yields $u\in\mathcal B_N$ such that
$$
\mathcal T(u)=u.
$$

We set $\rho=\mathcal R(u)$. Then Lemma~\ref{estimate on rho lemma} gives the
continuity equation \eqref{Navier Stokes}$_1$, while Lemma~\ref{estimate on u lemma}, together with the
fixed-point identity, gives the momentum equation \eqref{Navier Stokes}$_2$. Thus $(\rho,u)$ is
a solution of (NSENBC). The bounds in \eqref{boundedness on rho local-in-time solution} and \eqref{boundedness on u local-in-time solution}
follow directly from Lemma~\ref{estimate on rho lemma} and
Lemma~\ref{estimate on u lemma}, and their constants are independent of $K$.

The Dirichlet problem is obtained by replacing the Navier velocity
operator $\mathcal V$ by the corresponding Dirichlet operator
$\mathcal V_D$. The same fixed-point argument yields
$(\tilde\rho,\tilde u)$ on the same time interval $[0,T]$ with the
stated bounds.
\end{proof}

\section{Convergence of smooth solutions: Proof of Theorem~\ref{convergence smooth thm}}\label{convergence section}

In this section we give the proof of Theorem~\ref{convergence smooth thm}. We make use of the $K$-independent bounds obtained in Theorem~\ref{local-in-time existence theorem} in controlling both the densities and velocities. To begin with, suppose that $(\rho,\vu)$ and $(\trho,\tu)$ are smooth classical solutions to the system \eqref{Navier Stokes} which are defined on $\Omega \times [0, T]$ with boundary conditions \eqref{Navier condition} and \eqref{no slip condition} respectively satisfying the bounds \eqref{boundedness on rho local-in-time solution} and \eqref{boundedness on u local-in-time solution}, and assume that $(\rho,\vu)$ and $(\trho,\tu)$ have the same initial data $(\rho_0,\vu_0)$ verifying \eqref{boundedness condition on initial data} and \eqref{compatibility condition on u0 no-slip}. Define
\begin{align*}
\vw:=\vu-\tu,\quad\phi:=\rho-\trho.
\end{align*}
Then we have $\vw\big|_{t=0}=0$ and $\phi\big|_{t=0}=0$. Moreover, for all $t\in[0,T]$, $\vw$ and $\phi$ satisfy the following integral equations respectively
\begin{equation}\label{identity for phi}
\frac{1}{2}\intox\phi^2dx+\intoxt\phi\divv(\rho \vu-\trho\tu)dxds=0,
\end{equation}
\begin{align}\label{energy equality for w}
&\intoxt(\rho \vu-\trho\tu)_t\cdot \vw dxds+\intoxt\divv(\rho \vu\otimes \vu-\trho\tu\otimes\tu)\cdot \vw dxds\notag\\
&\quad+\intoxt(\nabla P(\rho)-\nabla P(\trho))\cdot \vw dxds+\intoxt\big(\mu|\nabla \vw|^2+\lambda(\divv \vw)^2\big)dxds\notag\\
&=-\intoxtb(K^{-1}\vu-\mu\No\cdot\nabla\tu)\cdot \vw dS_xds.
\end{align}
With the help of those $K$-independent bounds obtained in Theorem~\ref{local-in-time existence theorem}, we can bound $\phi$ and $\vw$ which will be given in subsequent lemmas.

We first prove the following lemma which gives an estimate on $\phi$.
\begin{lemma}\label{estimate on phi lemma}
For all $t\in[0,T]$, there exists a positive constant $C$ depending on $T, M_0, M_0'$, and $N$ as described in Theorem~\ref{local-in-time existence theorem} but independent of $K$ such that
\begin{align}\label{estimate on phi}
\intox\phi^2(x,t)dx&\le C\intoxt\phi^2dxds+C\bigg(\intoxt\phi^2dxds\bigg)^\frac{1}{2}\bigg(\intoxt|\nabla \vw|^2dxds\bigg)^\frac{1}{2}\notag\\
&\quad+C\bigg(\intoxt\phi^2dxds\bigg)^\frac{1}{2}\bigg(\intoxt|\vw|^2dxds\bigg)^\frac{1}{2}.
\end{align}
\end{lemma}

\begin{proof}
In view of \eqref{identity for phi}, since $\rho \vu-\trho\tu=\phi \vu+\trho \vw$, we can decompose the integral $\dis\intoxt\phi\divv(\rho \vu-\trho\tu)dxds$ as follows
\begin{align}\label{z1}
\intoxt\phi\divv(\rho \vu-\trho\tu)dxds=\intoxt\phi\divv(\phi \vu+\trho \vw)dxds:=I_1+I_2,
\end{align}
where
\begin{align*}
I_1:=\intoxt\phi\divv(\phi \vu)dxds,\quad I_2:=\intoxt\phi\divv(\trho \vw)dxds.
\end{align*}
Owing to
\begin{align*}
\intoxt\phi\divv(\phi \vu)dxds=\frac{1}{2}\intoxt\divv(\phi^2 \vu)dxds+\frac{1}{2}\intoxt\phi^2\divv\vu dxds,
\end{align*}
we obtain from \eqref{boundedness on u local-in-time solution}, \eqref{bound on L infty norm}, and \eqref{Navier condition} that
\begin{align}\label{bound on I1}
I_1&=\frac{1}{2}\intoxtb\phi^2 \vu\cdot \No dS_xds+\frac{1}{2}\intoxt\phi^2\divv\vu dxds\notag\\
&\le \frac{1}{2}\|\divv\vu\|_{L_{t,x}^\infty}\intoxt\phi^2dxds\le C\intoxt\phi^2dxds.
\end{align}
For the term $I_2$, one readily derives from \eqref{boundedness on rho local-in-time solution}, \eqref{boundedness on u local-in-time solution}, and \eqref{bound on L infty norm} that
\begin{align*}
I_2&\le C\intoxt|\trho||\nabla \vw||\phi|dxds+C\intoxt|\nabla\trho||\vw||\phi|dxds\notag\\
&\le C\|\trho\|_{L_{t,x}^\infty}\intoxt|\nabla \vw||\phi|dxds+C\int_0^t\|\nabla\trho\|_{L^4}\|\vw\|_{L^4}\|\phi\|_{L^2}ds.
\end{align*}
Using \eqref{boundedness on rho local-in-time solution}, the term $\dis C\|\trho\|_{L_{t,x}^\infty}\intoxt|\nabla \vw||\phi|dxds$ can be bounded by $$C\bigg(\intoxt\phi^2dxds\bigg)^\frac{1}{2}\bigg(\intoxt|\nabla \vw|^2dxds\bigg)^\frac{1}{2},$$ and for $\dis C\int_0^t\|\nabla\trho\|_{L^4}\|\vw\|_{L^4}\|\phi\|_{L^2}ds$, we apply the embedding \eqref{embedding from Lp to H1} and the bound \eqref{boundedness on rho local-in-time solution} to get
\begin{align*}
&C\int_0^t\|\nabla\trho\|_{L^4}\|\vw\|_{L^4}\|\phi\|_{L^2}ds\\
&\le C\int_0^t\Big(\|\nabla\trho\|_{L^2}
+\|\nabla\trho\|_{L^2}^\frac{1}{4}\|\nabla^2\trho\|_{L^2}^\frac{1}{4}\Big)
\Big(\|\vw\|_{L^2}+\|\vw\|_{L^2}^\frac{3}{4}\|\nabla\vw\|_{L^2}^\frac{1}{4}\Big)
\|\phi\|_{L^2}ds\\
&\le C\bigg(\intoxt\phi^2dxds\bigg)^\frac{1}{2}\bigg(\intoxt|\nabla \vw|^2dxds\bigg)^\frac{1}{2}
+C\bigg(\intoxt\phi^2dxds\bigg)^\frac{1}{2}\bigg(\intoxt|\vw|^2dxds\bigg)^\frac{1}{2}.
\end{align*}
Hence, we have
\begin{align}\label{bound on I2}
I_2&\le C\bigg(\intoxt\phi^2dxds\bigg)^\frac{1}{2}\bigg(\intoxt|\nabla \vw|^2dxds\bigg)^\frac{1}{2}
+C\bigg(\intoxt\phi^2dxds\bigg)^\frac{1}{2}\bigg(\intoxt|\vw|^2dxds\bigg)^\frac{1}{2}.
\end{align}
Inserting \eqref{bound on I1} and \eqref{bound on I2} into \eqref{z1} and applying \eqref{identity for phi}, then the assertion \eqref{estimate on phi} follows.
\end{proof}

Next we prove the following lemma which consists of the estimate on $\vw$.
\begin{lemma}\label{estimate on w lemma}
For all $t\in[0,T]$, there exists a positive constant $C$ depending on $T, M_0, M_0'$, and $N$ as described in Theorem~\ref{local-in-time existence theorem} but independent of $K$ such that
\begin{align}\label{estimate on w}
&\intox\rho|\vw|^2dx+\intoxt|\nabla \vw|^2dxds\notag\\
&\le C\intoxt|\vw|^2dxds+C\bigg(\intoxt|\vw|^2dxds\bigg)^\frac{1}{2}
\bigg(\intoxt|\nabla \vw|^2dxds\bigg)^\frac{1}{2}+C\bigg(\intoxt|\nabla \vw|^2dxds\bigg)^\frac{1}{2}\notag\\
&\,\,\,\times\bigg(\intoxt\phi^2dxds\bigg)^\frac{1}{2}+C(\|\tu_t\|_{H^2}+1)\bigg(\intoxt|\vw|^2dxds\bigg)^\frac{1}{2}
\bigg(\intoxt\phi^2dxds\bigg)^\frac{1}{2}\notag\\
&\,\,\,+C\intoxtb|\vu|^2dS_xds+C\bigg(\intoxtb|\vu|^2dS_xds\bigg)^\frac{1}{2}.
\end{align}
\end{lemma}
\begin{proof}
Define
\begin{align*}
I_3&:=\intoxt(\rho \vu-\trho\tu)_t\cdot \vw dxds,\quad I_4:=\intoxt\divv(\rho \vu\otimes \vu-\trho\tu\otimes\tu)\cdot \vw dxds\\
I_5&:=\intoxt(\nabla P(\rho)-\nabla P(\trho))\cdot \vw dxds.
\end{align*}
Noting that
\begin{align*}
\intoxt(\rho \vu-\trho\tu)_t\cdot \vw dxds
=\frac{1}{2}\intoxt(\rho|\vw|^2)_tdxds+\frac{1}{2}\intoxt\rho_t|\vw|^2dxds
+\intoxt(\phi\tu)_t\cdot \vw dxds,
\end{align*}
one has
\begin{align}\label{eqn for I3}
I_3=\frac{1}{2}\intox\rho|\vw|^2dx+I_{3,1}+I_{3,2}+I_{3,3},
\end{align}
where
\begin{align*}
I_{3,1}:=\frac{1}{2}\intoxt\rho_t|\vw|^2dxds,\,\,\, I_{3,2}:=\intoxt\phi_t\tu\cdot \vw dxds,\,\,\,I_{3,3}:=\intoxt\phi\tu_t\cdot \vw dxds.
\end{align*}
In view of \eqref{energy equality for w}, we thus find that
\begin{align}\label{z2}
\frac{1}{2}\intox\rho|\vw|^2dx+\intoxt\big(\mu|\nabla \vw|^2+\lambda(\divv \vw)^2\big)dxds&=-I_{3,1}-I_{3,2}-I_{3,3}-I_{4}-I_{5} \notag \\
& \quad -\intoxtb(K^{-1}\vu-\mu\No\cdot\nabla\tu)\cdot \vw dS_xds.
\end{align}

The term $I_{3,3}$ is readily bounded by
\begin{align}\label{bound on I4}
|I_{3,3}|\le C\|\tu_t\|_{H^2}\bigg(\intoxt|\vw|^2dxds\bigg)^\frac{1}{2}
\bigg(\intoxt\phi^2dxds\bigg)^\frac{1}{2}.
\end{align}
For $I_{3,1}$, we use the mass equation \eqref{Navier Stokes}$_1$ for $(\rho,\vu)$ and notice that $\vu\cdot \No=0$ on $\partial\Omega$, we deduce from \eqref{boundedness on rho local-in-time solution} and \eqref{boundedness on u local-in-time solution} that
\begin{align}
|I_{3,1}|&\le C\intoxt|\rho \vu||\nabla \vw||\vw|dxds \notag \\
&\le C\|\rho \vu\|_{L_{t,x}^\infty}\bigg(\intoxt|\vw|^2dxds\bigg)^\frac{1}{2}\bigg(\intoxt|\nabla \vw|^2dxds\bigg)^\frac{1}{2} \notag \\
&\le C\bigg(\intoxt|\vw|^2dxds\bigg)^\frac{1}{2}\bigg(\intoxt|\nabla \vw|^2dxds\bigg)^\frac{1}{2},
\end{align}
while for the term $I_{3,2}$, we notice that
\begin{align*}
\intoxt\phi_t\tu\cdot \vw dxds
=-\intoxt\divv(\rho \vu-\trho\tu)\tu\cdot \vw dxds=-\intoxt\divv(\phi \vu+\trho \vw)\tu\cdot \vw dxds
\end{align*}
and recall the fact that $\tu=0$ on $\partial\Omega$ to obtain
\begin{align*}
|I_{3,2}|&=\bigg|\intoxt\divv(\rho \vu-\trho\tu)\tu\cdot \vw dxds\bigg|= \bigg|\intoxt(\phi \vu+\trho \vw)\cdot\nabla(\tu\cdot \vw)dxds\bigg|,
\end{align*}
which gives
\begin{align}
|I_{3,2}|&\le C\intoxt|\vw|^2dxds+C\bigg(\intoxt|\vw|^2dxds\bigg)^\frac{1}{2}\bigg(\intoxt|\nabla \vw|^2dxds\bigg)^\frac{1}{2}\notag \\
&\quad+C\bigg(\intoxt|\vw|^2dxds\bigg)^\frac{1}{2}
\bigg(\intoxt\phi^2dxds\bigg)^\frac{1}{2}\notag \\
&\quad+C\bigg(\intoxt|\nabla \vw|^2dxds\bigg)^\frac{1}{2}\bigg(\intoxt\phi^2dxds\bigg)^\frac{1}{2}.
\end{align}
Upon integrating by parts, using the boundary condition $\vu\cdot \No=0$ on $\partial \Omega$ and applying bounds \eqref{boundedness on rho local-in-time solution}--\eqref{boundedness on u local-in-time solution}, the term $I_4$ can be estimated as follows (summation over repeated indexes is understood)
\begin{align}
|I_4|&\le \bigg|\intoxt\nabla w^j\cdot(\rho u^j \vu-\trho\tilde{u}^{j} \tu)dxds\bigg|\notag\\
&\le C\Big(\intoxt|\nabla \vw|^2dxds\Big)^\frac{1}{2}\Big(\intoxt\phi^2dxds\Big)^\frac{1}{2}
+C\Big(\intoxt|\nabla \vw|^2dxds\Big)^\frac{1}{2}\Big(\intoxt|\vw|^2dxds\Big)^\frac{1}{2}.
\end{align}
For the term $I_5$, we again integrate by parts and use the boundary condition $\vu\cdot \No=0$ on $\partial \Omega$ to obtain
\begin{align}\label{bound on I5}
|I_5|&\le\bigg|\intoxtb \vu\cdot \No(P(\rho)-P(\trho))dS_xds\bigg|+\bigg|\intoxt(P(\rho)-P(\trho))\cdot\nabla \vw dxds\bigg|\notag\\
&\le C\bigg(\intoxt|\nabla \vw|^2dxds\bigg)^\frac{1}{2}\bigg(\intoxt|\rho-\trho|^2dxds\bigg)^\frac{1}{2}.
\end{align}
Recalling that $\vw=\vu$ on $\partial\Omega$ and applying \eqref{bound on bdry integral} on $\nabla \tu$ to get
\begin{align}\label{bound on bdry terms}
&-\intoxtb(K^{-1}\vu-\No\cdot\nabla\tu)\cdot \vw dS_xds\notag\\
&=-\intoxtb K^{-1}|\vu|^2dS_xds + \intoxtb (\No\cdot\nabla\tu)\cdot \vu dS_xds\notag\\
&\le C\bigg(\intoxtb|\nabla\tu|^2dS_xds\bigg)^\frac{1}{2}
\bigg(\intoxtb|\vu|^2dS_xds\bigg)^\frac{1}{2}\notag\\
&\le C\bigg(\intoxt\big(|\nabla\tu|^2+|\nabla\tu||\Delta\tu|\big)dxds\bigg)^\frac{1}{2}
\bigg(\intoxtb|\vu|^2dS_xds\bigg)^\frac{1}{2}\notag\\
&\le C\bigg(\intoxtb|\vu|^2dS_xds\bigg)^\frac{1}{2}.
\end{align}
Thus, substituting \eqref{bound on I4}--\eqref{bound on bdry terms} into \eqref{z2}, we conclude \eqref{estimate on w}.
\end{proof}

The following lemma contains the crucial bound on $\|\vu\|_{L^2(\partial\Omega)}$ in terms of $K$ which will be used for proving Theorem~\ref{convergence smooth thm}.
\begin{lemma}\label{new estimate on u lemma}
There exists $T^*\in(0,T]$ such that for all $t\in[0,T^*]$,
\begin{align}\label{L2 bound on u on bdry}
\intoxtb |\vu|^2 dS_xds\le M_0K.
\end{align}
\end{lemma}
\begin{proof}
Multiplying the momentum equation \eqref{Navier Stokes}$_2$ by $u^j$, summing over $j$, and integrating by parts to obtain
\begin{align}\label{equality from momentum eqn}
&\intox\frac{\rho|\vu|^2}{2}dx+\intoxt \vu\cdot\nabla (P-\bar{P})dxds+\intoxt\big(\mu|\nabla \vu|^2+\lambda(\divv\vu)^2\big)dxds\notag\\
&=\intox\frac{\rho_0|\vu_0|^2}{2}dx+\intoxtb\mu(\nabla \vu\No)\cdot \vu dS_xds+\lambda\intoxtb\divv\vu(\vu\cdot \No)dS_xds\notag\\
&=\intox\frac{\rho_0|\vu_0|^2}{2}dx-K^{-1}\intoxtb|\vu|^2 dS_xds,
\end{align}
where the last equality of \eqref{equality from momentum eqn} follows from the boundary condition \eqref{Navier condition}. Upon integrating by parts and using the boundary condition $\vu\cdot \No=0$ on $\partial\Omega$, we have
\begin{align*}
\bigg|\intoxt \vu\cdot\nabla Pdxds\bigg|&\le \bigg(\intoxt|\nabla \vu|^2dxds\bigg)^\frac{1}{2}\bigg(\intoxt |P-\bar{P}|^2dxds\bigg)^\frac{1}{2}\\
&\le C\bigg(\intoxt|\nabla \vu|^2dxds\bigg)^\frac{1}{2}\bigg(\intoxt|\rho-\bar{\rho}|^2dxds\bigg)^\frac{1}{2},
\end{align*}
which together with \eqref{equality from momentum eqn} yields that
\begin{align}\label{L2 estimate on u}
&\intox\frac{\rho|\vu|^2}{2}dx+\intoxt\big(\mu|\nabla \vu|^2+\lambda(\divv\vu)^2\big)dxds\notag\\
&\le C\bigg(\intoxt|\nabla \vu|^2dxds\bigg)^\frac{1}{2}\bigg(\intoxt|\rho-\bar{\rho}|^2dxds\bigg)^\frac{1}{2}
+\intox\frac{\rho_0|\vu_0|^2}{2}dx-\intoxtb K^{-1}|\vu|^2 dS_xds.
\end{align}
Make use of the boundary condition \eqref{Navier condition} and the bound \eqref{boundedness on rho local-in-time solution} on $\rho$ to obtain
\begin{align}\label{L2 estimate on rho}
\intox\frac{|\rho-\brho|^2}{2}dx&=\intox\frac{|\rho_0-\brho|^2}{2}dx+\intoxt \rho u\cdot\nabla\rho dxds\notag\\
&\le \intox\frac{|\rho_0-\brho|^2}{2}dx+\intoxt\rho|\nabla\rho||\vu|dxds\notag\\
&\le \intox\frac{|\rho_0-\brho|^2}{2}dx+C\bigg(\intoxt|\nabla\rho|^2dxds\bigg)^\frac{1}{2}
\bigg(\intoxt|\vu|^2dxds\bigg)^\frac{1}{2}.
\end{align}
We sum up \eqref{L2 estimate on u} and \eqref{L2 estimate on rho}, recall the $H^2$-bound \eqref{boundedness on rho local-in-time solution} on $\rho-\brho$ and choose $T^*\in(0,T]$ small enough, then for $t\in[0,T^*]$, we have
\begin{align*}
\sup_{s\in[0,t]}\intox\big(\rho^2+|\vu|^2\big)dx
\le\intox\frac{\rho_0|\vu_0|^2}{2}dx+\intox\frac{\rho_0^2}{2}dx-\intoxtb K^{-1}|\vu|^2 dS_xds,
\end{align*}
which implies \eqref{L2 bound on u on bdry}.
\end{proof}

Now we are ready to give the proof of Theorem~\ref{convergence smooth thm}.
\begin{proof}[Proof of Theorem~\ref{convergence smooth thm}]
We sum up \eqref{estimate on phi} and \eqref{estimate on w}, employ the positive lower bound \eqref{boundedness on rho local-in-time solution} on $\rho$, and apply Young's inequality to obtain that, for $t\in[0,T]$,
\begin{align}\label{L2 bound on phi and w}
&\intox\big(\phi^2(x,t)+|\vw|^2(x,t)\big)dx+\intoxt|\nabla \vw|^2dxds\notag\\
&\le C(\|\tu_t\|_{H^2}+1)\intoxt\big(\phi^2+|\vw|^2\big)dxds
+C\intoxtb|\vu|^2dS_xds+C\bigg(\intoxtb|\vu|^2dS_xds\bigg)^\frac{1}{2}.
\end{align}
Applying Gr\"{o}nwall's inequality on \eqref{L2 bound on phi and w} and using the bound \eqref{boundedness on u local-in-time solution} on the time integral $\dis\int_0^T\|\tu_t(\cdot,t)\|_{H^2}^2dt$, we get that, for all $t\in[0,T^*]$,
\begin{align}\label{estimate on L2 norm of phi and w step 1}
\intox\big(\phi^2(x,t)+|\vw|^2(x,t)\big)dx
\le C\left[\intoxtb|\vu|^2dS_xds+\bigg(\intoxtb|\vu|^2dS_xds\bigg)^\frac{1}{2}\right]e^{Ct}.
\end{align}
Let $T^*>0$ be chosen as in Lemma~\ref{estimate on u lemma}. We apply \eqref{L2 bound on u on bdry} on \eqref{estimate on L2 norm of phi and w step 1} to obtain that, for all $t\in[0,T^*]$,
\begin{align}\label{estimate on L2 norm of phi and w step 2}
\intox\big(\phi^2(x,t)+|\vw|^2(x,t)\big)dx\le C\Big(M_0K+M_0^\frac{1}{2}K^\frac{1}{2}\Big)e^{Ct}.
\end{align}
Hence, by taking $K\to0$ in \eqref{estimate on L2 norm of phi and w step 2}, we have that
\begin{align}\label{convergence of phi and u in L2}
\begin{cases}
\rho\to\trho\mbox{ in $L^\infty([0,T^*];L^2(\Omega))$ as $K\to0$,} \\
\vu\to\tu\mbox{ in $L^\infty([0,T^*];L^2(\Omega))$ as $K\to0$.}
\end{cases}
\end{align}
Moreover, by applying the convergences given in \eqref{convergence of phi and u in L2} on \eqref{estimate on L2 norm of phi and w step 2}, we have
\begin{align*}
\vu\to\tu\mbox{ in $L^2([0,T^*];H^1(\Omega))$ as $K\to0$.}
\end{align*}
Since $\tu=0$ on $\partial\Omega$, using \eqref{L2 bound on u on bdry} from Lemma~\ref{estimate on u lemma}, we conclude that $\vu\to\tu$ in $L^2([0,T^*];L^2(\partial\Omega))$ as $K\to0$ and \eqref{convergence of u on the boundary} follows. Finally, by Sobolev's inequality, for $s_1\in(0,2]$, there exists $\sigma=\sigma(s_1)\in(0,1)$ such that for $t\in[0,T^*]$,
\begin{align}\label{Hs estimate on u}
\|(\vu-\tu)(\cdot,t)\|_{H^{s_1}}\le\|(\vu-\tu)(\cdot,t)\|_{L^2}^\sigma\|(\vu-\tu)(\cdot,t)\|_{H^{3}}^{1-\sigma},
\end{align}
so that by applying the bound \eqref{boundedness on u local-in-time solution} on $\|(\vu-\tu)(\cdot,t)\|_{H^{3}}^{1-\sigma}$, we have
\begin{align*}
\mbox{$\dis\sup_{0\le t\le T^*}\|(\vu-\tu)(\cdot,t)\|_{H^{s_1}}\to0$ as $K\to0$}
\end{align*}
while for $s_1=0$, we apply the already established $L^2$-convergence directly, and hence \eqref{convergence of u in Hs} follows. The proof of \eqref{convergence of rho in Hs} is just similar and we finish the proof of Theorem~\ref{convergence smooth thm}.
\end{proof}

\section{Convergence of weak solutions: Proof of Theorem~\ref{convergence weak thm}}\label{weak solution section}

In this section we address the convergence of weak solutions to \eqref{Navier Stokes}--\eqref{initial condition} as $K$ vanishes and give the proof of Theorem~\ref{convergence weak thm}. To this end, we make use of a Lagrangian framework for comparing weak solutions. The following lemma shows the existence of particle trajectories $\tX$ and $X$ for the velocities $\tu$ and $\vu$, respectively.

\begin{lemma}\label{trajectory lemma}
Let $\tu$, $\vu$ satisfy \eqref{new spaces for weak sol 1} and \eqref{new spaces for weak sol 1'}. Then there are unique functions $\tX$, $X\in C(\bar{\Omega}\times[0,T]^2)$ satisfying
\begin{align}\label{integral curve X 1}
\begin{cases}
\dis\frac{\partial \tX}{\partial t}(y,t,t')
=\tu(\tX(y,t,t'),t),\\ X(y,t',t')=y
\end{cases}
\end{align}
and
\begin{align}\label{integral curve X 2}
\begin{cases}
\dis\frac{\partial X}{\partial t}(y,t,t')
=\vu(X(y,t,t'),t),\\ X(y,t',t')=y.
\end{cases}
\end{align}
The functions $\tX(\cdot,t,t')$ and $X(\cdot,t,t')$ are Lipschitz on $\bar{\Omega}$ for $(t,t')\in[0,T]^2$, and there is a positive constant $C$ independent of $K$ such that
\begin{equation*}
\Big\|\frac{\partial \tX}{\partial y}(\cdot,t,t')\Big\|_{L^\infty(\bar{\Omega})},\ \ \Big\|\frac{\partial X}{\partial y}(\cdot,t,t')\Big\|_{L^\infty(\bar{\Omega})}\le C,\qquad (t,t')\in[0,T]^2.
\end{equation*}
Moreover, for each $(t,s)\in [0,T]^2$, the maps $y\mapsto \tX(y,t,s)$ and $y\mapsto X(y,t,s)$ are homeomorphisms of $\bar{\Omega}$ into $\bar{\Omega}$ and leaving $\partial\Omega$ invariant.
\end{lemma}
\begin{proof}
We give the proof for $\tX$ and the case for $X$ is just similar. For each $(y,t,s)\in\bar{\Omega}\times[0,T]^2$ with $y=(y_1,y_2,y_3)$, define $\tX(y,t,s)=(\tX^1,\tX^2,\tX^3)(y,t,s)$ by
\begin{align*}
\tX^{i}(y,t,s)=y_{i}+\int_{s}^{t}\tilde u^{i}(\tX^{i}(y,\tau,s),\tau)d\tau,\qquad \mbox{for $i\in\{1,2\}$},
\end{align*}
and
\begin{align*}
\tX^{3}(y,t,s)=
\begin{cases}
\dis y_{i}+\int_{s}^{t}\tilde{u}^{i}(\tX^{i}(y,\tau,s),\tau)d\tau,&\mbox{if $y_3\neq0$},\\
0,&\mbox{if $y_3=0$}.
\end{cases}
\end{align*}
Then direct calculations show that $\tX\in C(\bar{\Omega}\times[0,T]^2)$ and $\tX$ satisfies \eqref{integral curve X 1}. By \eqref{integral curve X 1}, we have
\begin{align*}
\frac{1}{2}\frac{\partial}{\partial t}\bigg|\frac{\partial \tX}{\partial y}(y,t,s)\bigg|^2\le \|\nabla\tu(\cdot,t)\|_{L^\infty(\bar{\Omega})}|\frac{\partial \tX}{\partial y}(y,t,s)|^2,
\end{align*}
which combined with $\tu\in L^1((0,T);W^{1,\infty}(\bar{\Omega}))$ and \eqref{new spaces for weak sol 1'} yields that, there is a constant $C>0$ independent of $K$ such that,
\begin{equation*}
\Big\|\frac{\partial \tX}{\partial y}(\cdot,t,t')\Big\|_{L^\infty(\bar{\Omega})}\le C,\qquad (t,t')\in[0,T]^2.
\end{equation*}
Moreover, since $\tu^{3}=0$ on $\partial\Omega$, the map $\tX$ leaves $\partial\Omega$ invariant. This completes the proof of Lemma~\ref{trajectory lemma}.
\end{proof}

The next lemma shows how the Jacobian determinants of $\tX(\cdot,t,s)$ and $X(\cdot,t,s)$ can be computed using density functions $\trho$ and $\rho$. The proof can be found in \cite{hoff06} and we omit the details for simplicity.
\begin{lemma}
Let $\tu$, $\vu$ satisfy \eqref{new spaces for weak sol 1} and \eqref{new spaces for weak sol 1'}. Assume that $\trho\tu$ and $\rho\vu$ are locally integrable on $\Omega\times[0,T]$ and the weak form of mass equation holds for both the pairs $(\trho,\tu)$ and $(\rho,\vu)$. Then, for any $t_0\in[0,T]$, we have
\begin{gather}\label{Jacob of tX}
\trho(\tX(\cdot,t,t_0),t)|\det\nabla_y\tX(\cdot,t,t_0)|=\trho(\cdot,t_0)\qquad\mbox{a.e. on $\Omega$},
\\ \label{Jacob of X}
\rho(X(\cdot,t,t_0),t)|\det\nabla_y X(\cdot,t,t_0)|=\rho(\cdot,t_0)\qquad\mbox{a.e. on $\Omega$}.
\end{gather}
\end{lemma}

Let $\tX$, $X$ be the functions satisfying \eqref{integral curve X 1} and \eqref{integral curve X 2}, respectively. Define $S(x,t)$, $S^{-1}(x,t)$ by
\begin{gather}\label{def of S}
S(x,t)=X(\tX(x,0,t),t,0),
\\ \label{def of S-1}
S^{-1}(x,t)=\tX(X(x,0,t),t,0).
\end{gather}
In the following lemma, we give some properties of $S$ and $S^{-1}$ which will be crucial for later analysis.
\begin{lemma}\label{prop on S}
Let $S$ and $S^{-1}$ be as given in \eqref{def of S}--\eqref{def of S-1}. Then we have:

\noindent{\bf (a)} $S^{\pm1}$ is continuous on $\bar{\Omega}\times[0,T]$ and Lipschitz continuous on $\bar{\Omega}\times[\tau,T]$ for all $\tau>0$, and there is a positive constant $C$ such that $$\|\nabla S^{\pm1}(\cdot,\tau)\|_{L^\infty}\le C,\qquad t\in[0,T];$$

\noindent{\bf (b)} $(S_t+\nabla S \tu)(x,t)=\vu(S(x,t),t)$ a.e. in $\bar{\Omega}\times(0,T);$

\noindent{\bf (c)} $\trho(S(x,t),t)\rho_0(X(x,0,t))\det\nabla S(x,t)=\rho(x,t)\trho_0(X(x,0,t))$ a.e. in $\Omega\times(0,T)$;

\noindent{\bf (d)} For all $t\in(0,T)$,
\begin{align}
\intox|x-S(x,t)|^2dx&\le Ct\int_0^t\intox|\tu(x,\tau)-\vu(S(x,\tau),\tau)|^2dxd\tau,\label{bound on S}\\
\intox|x-S^{-1}(x,t)|^2dx&\le Ct\int_0^t\intox|\tu(S^{-1}(x,\tau),\tau)-\vu(x,\tau)|^2dxd\tau.\label{bound on S -1}
\end{align}
\end{lemma}
\begin{proof}
The desired results follow directly from Lemmas 2.1 and 2.2.
\end{proof}

Now we are in a position to give the proof of Theorem~\ref{convergence weak thm}. In what follows, $C$ always denotes a generic positive constant which depends on the parameters $T$, $\bar{C}_1,\bar{C}_2,\bar{C}_3,\bar{C}_4,\bar{\rho}_1,\bar{\rho}_2$, $\bar{\rho}$, $\mu$, $\lambda$, and $a$ but independent of $K$.

{\it Proof of Theorem ~\ref{convergence weak thm}.}
Let $\psi:\bar\Omega\times[0,T]\rightarrow\R^3$ be a test function which is locally Lipschitz on $\Omega\times[0,T]$ satisfying
\begin{itemize}
\item $\psi\in (H^2\cap C^2)(\Omega\times[0,t])$;
\item $\psi$, $\psi_t\in L^2(\Omega\times(0,T))$ and $\nabla\psi \in L^2(\Omega\times(0,T))\cap L^\infty(\Omega\times(0,T))$;
\item $\psi(\cdot,T)=0$;
\item $\psi\cdot\No(x)=0$ on $\partial\Omega$.
\end{itemize}
In view of weak forms of \eqref{WF2} and \eqref{WF2 DBC}, we have
\begin{align}\label{weak form for rho}
-\intox \rho_0(x)\vu_0(x)\psi(x,0)dx
&=\int_0^T\intox \Big[\rho \vu\cdot(\psi_t+\nabla\psi \vu)+(P(\rho)-\bar{P})\divv(\psi)-\mu\nabla u^j\cdot\nabla\psi^j\notag \\
&\quad\quad\quad\quad\quad\quad-\lambda\divv(\vu)\divv(\psi)\Big]dxd\tau-\int_0^T\int_{\partial\Omega}K^{-1}\vu\cdot\psi dS_xd\tau,
\\ \label{weak form for bar rho}
-\intox \tilde{\rho}_0(x)\tilde{\vu}_0(x)\tilde{\psi}(x,0)dx
&=\int_0^T\intox \Big[\trho\tu\cdot(\tilde{\psi}_t+\nabla\tilde{\psi}\tu)+(P(\trho)-\bar{P})
\divv(\tilde{\psi})-\mu\nabla\bar{u}^j\cdot\nabla\tilde{\psi}^j\notag \\
&\quad\quad\quad\quad\quad\quad-\lambda\divv(\tu)\divv(\tilde{\psi})\Big]dxd\tau\notag\\
&\quad - \int_0^T\int_{\partial\Omega}\lambda\nabla \tilde{u}^3\cdot\tilde{\psi}dS_{x}d\tau - \int_0^T\int_{\partial\Omega}\mu\partial_{x_3}\tu\cdot\tilde{\psi}dS_{x}d\tau,
\end{align}
where $\tilde{\psi}=\psi\circ S^{-1}$ and $\bar{P}=P(\bar{\rho})$. Define a function $A_0$ by
\begin{align*}
A_0\rho=(\trho\circ S)|\det(\nabla S)|,
\end{align*}
then $A_0$ satisfies
\begin{align*}
A_0=\frac{\trho_0(X(x,0,t))}{\rho_0(X(x,0,t))}.
\end{align*}
Denote by $z=\vu-\tu\circ S$. Since $\tilde{\psi}(\cdot,0)=\psi(\cdot,0)$, we have
\begin{align}\label{estimate on weak form for u and rho}
&\intox (\trho_0\tu_0-\rho_0\vu_0)\cdot\psi(x,0)dx\notag\\
&=\int_0^T\intox \left[\rho z\cdot(\psi_t+\nabla\psi\vu)+(1-A_0)\rho(\tu\circ S)(\psi_t+\nabla\psi\vu)\right]dxd\tau\notag\\
&\quad+\int_0^T\intox \left[(P(\rho)-\bar{P})\divv(\psi)-(P(\tilde{\rho})-\bar{P})\divv(\tilde\psi)-\mu(\nabla u^j\nabla\psi^j-\nabla\tilde{u}^j\nabla\tilde{\psi}^j\right]dxd\tau\notag\\
&\quad-\int_0^T\intox \left[\lambda(\divv(\vu)\divv(\psi)-\divv(\tu)\divv(\tilde\psi)\right]dxd\tau\notag\\
&\quad-\int_0^T\intoxb K^{-1}\vu\cdot\psi dS_{x}d\tau + \int_0^T\int_{\partial\Omega}\lambda\nabla \tilde{u}^3\cdot\tilde{\psi}dS_{x}d\tau + \int_0^T\int_{\partial\Omega}\mu\partial_{x_3}\tu\cdot\tilde{\psi}dS_{x}d\tau.
\end{align}
We further rewrite the terms in \eqref{estimate on weak form for u and rho} as follows. Using $\tilde{F}$ and $\tilde{\omega}$ as defined in \eqref{def of F and omega}, we have
\begin{align*}
&\int_0^T\intox \left[(P(\tilde{\rho})-\bar{P})\divv(\tilde\psi) + \mu\nabla\tilde{u}^j\nabla\tilde{\psi}^j + \lambda\divv(\tu)\divv(\tilde\psi)\right]dxd\tau\\
&=\int_0^T\intox \left[(P(\tilde\rho)-\bar{P})\divv(\psi) + \mu\nabla u^j\nabla\psi^j + \lambda\divv(\vu)\divv(\psi)\right]dxd\tau\\
&\quad+\int_0^T\intox \left[\nabla\tilde{F}\cdot(\psi-\psi\circ S^{-1})+\mu\tilde{\omega}^{j,k}_{x_k}(\psi^j-\psi^j\circ S^{-1})\right]dxd\tau\\
&\quad+\int_0^T\intoxb\left[(\tilde{F}\tilde{\psi}\cdot\No + \mu\tilde{\omega}^{j,k}\tilde\psi^j\No^k)-(\tilde{F}\psi\cdot\No + \mu\tilde{\omega}^{j,k}\psi^j\No^k)\right]dS_{x}d\tau.
\end{align*}
Notice that $\tilde\psi\cdot\No=\psi\cdot\No=0$, we readily get
\begin{align*}
\int_0^T\intoxb\left[(\tilde{F}\tilde{\psi}\cdot\No + \mu\tilde{\omega}^{j,k}\tilde\psi^j\No^k)-(\tilde{F}\psi\cdot\No + \mu\tilde{\omega}^{j,k}\psi^j\No^k)\right]dS_{x}d\tau=\int_0^T\intoxb \mu(\tilde{\omega}^{j,k}\tilde\psi^j\No^k-\tilde{\omega}^{j,k}\psi^j\No^k)dS_{x}d\tau.
\end{align*}
Also, one finds that
\begin{align}\label{weak convergence step 1}
&\int_0^T\intox\left[\rho z\cdot(\psi_t+\nabla\psi\vu)-\mu(\nabla u^j-\nabla\tilde{u}^j)\cdot\nabla\psi^j-\lambda(\divv(\vu)-\divv(\tilde{\vu}))\divv(\psi)\right]dxd\tau\notag\\
&=\int_0^T\intox z\cdot\left[\rho(\psi_t+\nabla\psi\vu)+\mu\Delta\psi+\lambda\nabla\divv(\psi)\right]dxd
\tau
+\int_0^T\intox(\tilde{\vu}\circ S-\tilde{\vu})\cdot(\mu\Delta\psi + \lambda\nabla\divv(\psi))dxd\tau \notag\\
&\quad-\int_0^T\intoxb\mu(u^j - \tilde{u}^j)\No\cdot\nabla\psi^jdS_{x}d\tau - \lambda\int_0^T\intoxb\No\cdot(\vu-\tilde{\vu})\divv(\psi)dS_{x}d\tau.
\end{align}
Recall that $\No(x)=(0,0,-1)$. Since $u^3=0$ and $\tilde{\vu}=0$ holds on $\partial\Omega$ pointwisely, it implies
\begin{align*}
\lambda\int_0^T\intoxb\No\cdot(\vu-\tilde{\vu})\divv(\psi)dS_{x}d\tau=0,
\end{align*}
and the right-hand side of \eqref{weak convergence step 1} can be reduced to
\begin{align*}
&\int_0^T\intox z\cdot\left[\rho(\psi_t+\nabla\psi\vu)+\mu\Delta\psi+\lambda\nabla\divv(\psi)\right]dxd\tau\\
&\quad + \int_0^T\intox(\tilde{\vu}\circ S-\tilde{\vu})\cdot(\mu\Delta\psi + \lambda\nabla\divv(\psi))dxd\tau + \int_0^T\int_{\partial\Omega}\mu\partial_{x_3}\tu\cdot\tilde{\psi}dS_{x}d\tau.
\end{align*}
Hence, we obtain from \eqref{estimate on weak form for u and rho} that
\begin{align}\label{estimate on weak form for u and rho rewritten}
&\intox (\trho_0\tu_0-\rho_0\vu_0)\cdot\psi(x,0)dx\notag\\
&=\int_0^T\intox z\cdot\left[\rho(\psi_t+\nabla\psi\vu)+\mu\Delta\psi+\lambda\nabla\divv(\psi)\right]dxd\tau
\notag\\
&\quad+ \int_0^T\intox(\tilde{\vu}\circ S-\tilde{\vu})\cdot(\mu\Delta\psi + \lambda\nabla\divv(\psi))dxd\tau\notag\\
&\quad+\int_0^T\intox \left[(P(\rho)-P(\tilde\rho))\divv(\psi) + \mu\nabla u^j\nabla\psi^j + \lambda\divv(\vu)\divv(\psi)\right]dxd\tau\notag\\
&\quad+\int_0^T\intox \left[\nabla\tilde{F}\cdot(\psi-\psi\circ S^{-1})+\mu\tilde{\omega}^{j,k}_{x_k}(\psi^j-\psi^j\circ S^{-1})\right]dxd\tau\notag\\
&\quad-\int_0^T\intoxb K^{-1}\vu\cdot\psi dS_{x}d\tau + \int_0^T\int_{\partial\Omega}\lambda\nabla \tilde{u}^3\cdot\tilde{\psi}dS_{x}d\tau + \int_0^T\int_{\partial\Omega}\mu\partial_{x_3}\tu\cdot\tilde{\psi}dS_{x}d\tau\notag\\
&\quad+\int_0^T\intoxb \mu(\tilde{\omega}^{j,k}\tilde\psi^j\No^k-\tilde{\omega}^{j,k}\psi^j\No^k)dS_{x}d\tau.
\end{align}

Next we extend $\rho$, $\vu$ to be constant in time $t$ outside $[0,T]$ and let $\rho^\varepsilon$ and $\vu^\varepsilon$ be the corresponding smooth approximation obtained by mollifying in both $x$ and $t$. Then, for a given $G\in H^{\infty}(\R^3\times[0,T])$, we define $\psi^\varepsilon:\R^3\times[0,T]\to\R^3$ to be the solutions satisfying the system
\begin{align*}
\begin{cases} \rho^\varepsilon(\psi^\varepsilon_t+\vu^\varepsilon\cdot\nabla\psi^\varepsilon)+\mu\Delta\psi^\varepsilon+\lambda\nabla\divv(\psi^\varepsilon)=G,\\
((\psi^\varepsilon)^1,(\psi^\varepsilon)^2,(\psi^\varepsilon)^3)(x,\cdot) = \mu K(\partial_{x_3}(\psi^\varepsilon)^1,\partial_{x_3}(\psi^\varepsilon)^2,0)(x,\cdot),\qquad x\in\partial\Omega,\\
\psi^\varepsilon\cdot\No(x,\cdot)=0,\qquad x\in\partial\Omega,\\
\psi^\varepsilon(x,T)=0,\qquad x\in\Omega.
\end{cases}
\end{align*}
By the similar argument given in \cite{hoff06}, it can be shown that $\psi^\varepsilon$ satisfies the following estimate
\begin{align}\label{estimate on psi e}
&\sup_{0\le \tau\le T}\int_{\Omega}[|\psi^\varepsilon(x,\tau)|^2+|\nabla\psi^\varepsilon(x,\tau)|^2]dx + \int_0^T\int_{\Omega}[|\psi_t^\varepsilon+\nabla\psi^{\varepsilon}\vu^{\varepsilon}|^2+|D^2_x\psi^{\varepsilon}|^2]dxd\tau \notag\\
&\quad+ \frac{K^{-1}}{2}\int_0^T\int_{\partial\Omega}|\psi^\varepsilon|^2 dS_{x}d\tau\le C\int_0^T\int_{\Omega}|G|^2dxd\tau.
\end{align}

We will make use of $\psi^\varepsilon$ to estimate various terms involving $(\rho,\vu)$ and $(\trho,\tilde{\vu})$. Specifically, we choose $\psi=\psi^\varepsilon$ in \eqref{estimate on weak form for u and rho} and obtain
\begin{align}\label{identity on weak forms}
\intox (\trho_0\tu_0-\rho_0\vu_0)\cdot\psi^\varepsilon(x,0)dx
&=\int_0^T\intox z\cdot Gdxd\tau +\sum_{i=1}^5 \mathcal{R}_{i}- \int_0^T\intoxb K^{-1}\vu\cdot\psi^\varepsilon dS_{x}d\tau\notag \\
&\quad+ \int_0^T\intoxb\mu u^j \partial_{x_3}(\psi^\varepsilon)^jdS_{x}d\tau+\int_0^T\int_{\partial\Omega}\lambda\nabla \tilde{u}^3\cdot(\psi^\varepsilon\circ S^{-1}) dS_{x}d\tau \notag\\
&\quad + \int_0^T\int_{\partial\Omega}\mu\partial_{x_3}\tu\cdot(\psi^\varepsilon\circ S^{-1}) dS_{x}d\tau,
\end{align}
where $\mathcal{R}_i$ are defined by
\begin{align*}
&\mathcal{R}_1=\intoxT\big[\nabla \tilde{F}\cdot(\psi^\varepsilon-\psi^\varepsilon\circ S^{-1})+\mu\tilde{\omega}^{j,k}_{x_k}(\psi^\varepsilon-\psi^\varepsilon\circ S^{-1})\big]dxd\tau,\\
&\mathcal{R}_2=\intoxT(\tilde{\vu}\circ S-\tilde{\vu})\cdot(\mu\Delta\psi^\varepsilon+\lambda\nabla\divv(\psi^\varepsilon))dxd\tau,\\
&\mathcal{R}_3=\intoxT z\cdot \left[(\rho-\rho^\varepsilon)\psi^\varepsilon_t+\nabla\psi^\varepsilon(\rho \vu-\rho^\varepsilon \vu^\varepsilon)\right]dxd\tau,\\
&\mathcal{R}_4=\intoxT(1-A_0)\rho(\tilde{\vu}\circ S)\cdot(\psi^\varepsilon_t+\nabla\psi^\varepsilon \vu)dxd\tau,\\
&\mathcal{R}_5=\intoxT(P(\rho)-P(\tilde\rho))\divv(\psi^\varepsilon)dxd\tau.
\end{align*}
The left-hand side of \eqref{identity on weak forms} can be readily bounded by
\begin{align}\label{estimate of LHS}
\Big|\intox (\tilde{\rho}\tilde{\vu}_0-\rho_0\vu_0)\cdot\psi^\varepsilon(x,0)dx\Big|\le\|\rho_0\vu_0-\tilde{\rho}_0\tilde{\vu}_0\|_{L^2}\left(\intoxT|G|^2dxd\tau\right)^\frac{1}{2}.
\end{align}
By the similar procedures given in \cite{hoff06, suen21}, $\mathcal{R}_3$, $\mathcal{R}_4$, and $\mathcal{R}_5$ can be estimated as follows:
\begin{gather}\label{estimate on R3}
\lim_{\varepsilon\to0}\mathcal{R}_3=\lim_{\varepsilon\to0}\intoxT z\cdot \left[(\rho-\rho^\varepsilon)\psi^\varepsilon_t+\nabla\psi^\varepsilon(\rho u-\rho^\varepsilon u^\varepsilon)\right]dxd\tau=0,
\\ \label{estimate on R4}
|\mathcal{R}_4|\le C\|\|\rho_0-\trho_0\|_{L^2}\bigg(\intoxT|G|^2dxd\tau\bigg)^\frac{1}{2},
\\ \label{estimate on R5}
|\mathcal{R}_5|\le CT^\frac{1}{2}\left[\|\rho_0-\trho_0\|_{L^2}
+\bigg(\intoxT|z|^2dxd\tau\bigg)^\frac{1}{2}\right]
\bigg(\intoxT|G|^2dxd\tau\bigg)^\frac{1}{2}.
\end{gather}
For the terms $\mathcal{R}_1$ and $\mathcal{R}_2$, we define $\mathcal{R}_1'(t_1,t_2)$ and $\mathcal{R}_2'(t_1,t_2)$ respectively by
\begin{align*}
\mathcal{R}_1'(t_1,t_2)&=\int_{t_1}^{t_2}\int_{\Omega}\big[\nabla \tilde{F}\cdot(\psi^\varepsilon-\psi^\varepsilon\circ S^{-1})+\mu\tilde{\omega}^{j,k}_{x_k}(\psi^\varepsilon-\psi^\varepsilon\circ S^{-1})\big]dxd\tau\\
\mathcal{R}_2'(t_1,t_2)&=\int_{t_1}^{t_2}\int_{\Omega}(\tilde{\vu}\circ S-\tilde{\vu})\cdot(\mu\Delta\psi^\varepsilon+\lambda\divv(\psi^\varepsilon))dxd\tau,
\end{align*}
where $(t_1,t_2)\subseteq[0,T]$. Using H\"{o}lder's inequality and the estimate \eqref{estimate on psi e}, we have
\begin{align*}
|\mathcal{R}_1'(t_1,t_2)|&\le C\left(\int_{0}^{t_2}\int_{\Omega}|z|^2dxd\tau\right)^\frac{1}{2}\left(\int_{0}^{t_2}\int_{\Omega}|G|^2dxd\tau\right)^\frac{1}{2}\\
&\quad\times\left(\int_{t_1}^{t_2} \tau^\frac{4}{5}\|\nabla\tilde{F}(\cdot,\tau)\|^\frac{8}{5}_{L^4}+\|\nabla\tilde{\omega}(\cdot,\tau)\|^\frac{8}{5}_{L^4})d\tau\right)^\frac{5}{8},
\end{align*}
\begin{align*}
|\mathcal{R}_2'(t_1,t_2)|\le C\left(\int_{0}^{t_2}\int_{\Omega}|z|^2dxd\tau\right)^\frac{1}{2}
\left(\int_{0}^{t_2}\int_{\Omega}|G|^2dxd\tau\right)^\frac{1}{2}
\left(\int_{t_1}^{t_2} \tau\|\nabla\tu(\cdot,\tau)\|^2_{L^\infty}d\tau\right)^\frac{1}{2},
\end{align*}
and hence
\begin{align}\label{estimate on R1 and R2}
&|\mathcal{R}_1'(t_1,t_2)|+|\mathcal{R}_2'(t_1,t_2)|\notag\\
&\le C\left(\int_{0}^{t_2}\int_{\Omega}|z|^2dxd\tau\right)^\frac{1}{2}
\left(\int_{0}^{t_2}\int_{\Omega}|G|^2dxd\tau\right)^\frac{1}{2}\notag\\
&\quad\times\left[\left(\int_{t_1}^{t_2} \tau^\frac{4}{5}\|\nabla\tilde{F}(\cdot,\tau)\|^\frac{8}{5}_{L^4}+\|\nabla\tilde{\omega}(\cdot,\tau)\|^\frac{8}{5}_{L^4})d\tau\right)^\frac{5}{8}+\left(\int_{t_1}^{t_2} \tau\|\nabla\tu(\cdot,\tau)\|^2_{L^\infty}d\tau\right)^\frac{1}{2}\right].
\end{align}

For those boundary integrals that involve $\vu$, we make use of the boundary condition on $\psi^\varepsilon$ to get
\begin{align*}
&- \int_0^T\intoxb K^{-1}\vu\cdot\psi^\varepsilon dS_{x}d\tau + \int_0^T\intoxb\mu u^j \partial_{x_3}(\psi^\varepsilon)^jdS_{x}d\tau\\
&=- \int_0^T\intoxb K^{-1}\vu\cdot\psi^\varepsilon dS_{x}d\tau + \int_0^T\intoxb K^{-1}\vu\cdot\psi^\varepsilon dS_{x}d\tau=0,
\end{align*}
while for those boundary integrals that involve $\tilde{\vu}$, using Lemma~\ref{prop on S} and \eqref{estimate on psi e}, they can be bounded by
\begin{align*}
\left|\int_0^T\int_{\partial\Omega}\mu\partial_{x_3}\tu\cdot(\psi^\varepsilon\circ S^{-1}) dS_{x}d\tau\right|
&\le\bigg(\int_0^T\intoxb|\nabla\tilde{\vu}|^2dxd\tau\bigg)^\frac{1}{2}\bigg(\int_0^T\intoxb|\psi^\varepsilon\circ S^{-1}|^2dxd\tau\bigg)^\frac{1}{2}\\
&\le C\bigg(\int_0^T\intoxb|\psi^\varepsilon|^2dxd\tau\bigg)^\frac{1}{2}\notag \\
&\le CK^\frac{1}{2}\bigg(\int_0^T\intoxb|G|^2dxd\tau\bigg)^\frac{1}{2}.
\end{align*}
This along with \eqref{estimate of LHS}--\eqref{estimate on R5} implies that
\begin{align}\label{estimate on z in T}
\bigg|\intoxT z\cdot Gdxd\tau\bigg|
\le C\bigg[M_0\bigg(\intoxT|G|^2dxd\tau\bigg)^\frac{1}{2}
+|\mathcal{R}_1'(T,0)|+|\mathcal{R}_2'(T,0)|\bigg],
\end{align}
where $M_0$ is given by
\begin{align*}
M_0&=\|\rho_0-\bar{\rho}_0\|_{L^2}+\|\rho_0u_0-\bar{\rho}_0\bar{u}_0\|_{L^2}
+K^\frac{1}{2}+T^{\delta}\bigg(\intoxt|z|^2dxd\tau\bigg)^\frac{1}{2}
\end{align*}
for some $\delta>0$ independent of $K$. Following the analysis given in \cite[pp. 1758--1759]{hoff06}, there exists a small time $\tilde\tau>0$ such that
\begin{align*}
\left[\int_{t_1}^{t_2}\Big(\tau^\frac{4}{5}\|\nabla\tilde{F}(\cdot,\tau)
\|^\frac{8}{5}_{L^4}+\|\nabla\tilde{\omega}(\cdot,\tau)\|^\frac{8}{5}_{L^4}\Big)
d\tau\right]^\frac{5}{8}+\left(\int_{t_1}^{t_2} \tau\|\nabla\tu(\cdot,\tau)\|^2_{L^\infty}d\tau\right)^\frac{1}{2}\le\frac{1}{2C},
\end{align*}
and hence, using \eqref{estimate on R1 and R2}, it gives
\begin{equation*}
\bigg(\int_0^{\tilde\tau}\intox|z|^2dxd\tau\bigg)^\frac{1}{2}\le 2CM_0.
\end{equation*}
Consequently, one has that
\begin{equation*}
|\mathcal{R}_1'(0,\tilde\tau)|+|\mathcal{R}_2'(0,\tilde\tau)|\le M_0\bigg(\int_0^{\tilde\tau}\intox|G|^2dxd\tau\bigg)^\frac{1}{2}.
\end{equation*}
By applying \eqref{estimate on z in T} with $T$ replaced by $2\tilde\tau$, we get
\begin{equation*}
\bigg(\int_0^{2\tilde\tau}\intox|z|^2dxd\tau\bigg)^\frac{1}{2}\le 4CM_0.
\end{equation*}
Since $\tilde\tau>0$ is fixed, we can exhaust the interval $[0,T]$ in finitely many steps to obtain
\begin{equation*}
\bigg(\intoxT|z|^2dxd\tau\bigg)^\frac{1}{2}\le CM_0
\end{equation*}
for some new constant $C>0$. Thus, the term $\dis T^{\delta}\bigg(\intoxT|z|^2dxd\tau\bigg)^\frac{1}{2}$ can be eliminated from the definition of $M_0$. Therefore, we conclude that
\begin{align}
\bigg|\intoxT z\cdot Gdxd\tau\bigg|&\le CM_0\bigg(\intoxT|G|^2dxd\tau\bigg)^\frac{1}{2}.\label{L2 bound on u uniqueness}
\end{align}
Since the bound \eqref{L2 bound on u uniqueness} holds for any $G\in H^\infty(\R^3\times[0,T])$ and smooth bounded vector fields are dense in the relevant $L^2$ space, it shows that the term $\|z\|_{L^2([0,T]\times\Omega)}$ can be bounded by $M_0$. Finally, using the bound \eqref{spaces for weak sol 5} on the time integral on $\|\nabla\tu\|_{L^\infty}$, we have
\begin{align}\label{bound on bar u and bar u S}
\intoxT|\tu-\tu\circ S|^2dxd\tau\le\int_0^T\|\nabla\tu(\cdot,\tau)\|^2_{L^\infty}\intox|x-S(x,\tau)|^2dxd\tau
\le C\intoxT|z|^2dxd\tau,\notag
\end{align}
and hence \eqref{bound on difference} follows. This completes the proof of Theorem~\ref{convergence weak thm}. \hfill$\large\Box$



\end{document}